\documentclass[12pt, reqno]{amsart}

\usepackage{amsmath,amsfonts, amsthm, amssymb, color, cite}
\usepackage{color}
\usepackage{fancyhdr}
\usepackage{graphicx}
\usepackage{epstopdf}
\usepackage{lastpage}
\newtheorem{thm}{Theorem}[section]
\newtheorem{lma}{Lemma}[section]
\newtheorem{pro}{Proposition}[section]

\theoremstyle{definition}

\theoremstyle{remark}
\newtheorem{Remark}{Remark}[section]

\numberwithin{equation}{section}
\allowdisplaybreaks

\newcommand{\p}{\partial}
\newcommand{\e}{\varepsilon}

\def\f{\frac}

\def\hf1{^\f{1}{1-\xi^2}}

\newcommand{\vt}{V}
\newcommand{\ut}{U}

\newcommand{\X}{\mathcal{X}}
\newcommand{\Y}{\mathcal{Y}}

\def\be{\begin{equation}}
\def\en{\end{equation}}
\def\bs{\begin{split}}
\def\es{\end{split}}
\def\ba{\begin{align}}
\def\ea{\end{align}}

\author[Lin Chang* ]{Lin Chang* }
\address{School of Mathematical Sciences, Beihang University, Beijing, China}
\email{ changlin23@buaa.edu.cn}

\author[Lin He ]{ Lin He}
\address{College of Mathematics, Sichuan University, Chengdu, China}
\email{ lin\_he@scu.edu.cn}

\author[Jin Ma]{ Jin Ma}
\address{School of Mathematics and Information Science, Shandong Technology and Business University, Yantai 264005, P. R. China}
\email{  majingsinna@sina.com}

\renewcommand{\fancyhead}{}
\title{  S\MakeLowercase{tability} \MakeLowercase{of} \MakeLowercase{strong}   \MakeLowercase{viscous} \MakeLowercase{shock} \MakeLowercase{wave}  \MakeLowercase{under} \MakeLowercase{periodic} \MakeLowercase{perturbation} \MakeLowercase{for} 1-D  \MakeLowercase{isentropic} n\MakeLowercase{avier}-s\MakeLowercase{tokes} \MakeLowercase{system} \MakeLowercase{in} \MakeLowercase{the} \MakeLowercase{half} \MakeLowercase{space}}
\keywords{Impermeable wall problem, Large amplitude shock, Space-periodic perturbation, Asymptotic stability}
\date{\today}
\usepackage{subfig}
\usepackage{graphicx}
\begin{document}
\begin{abstract}

In this paper, a viscous shock wave under space-periodic perturbation
of 1-D isentropic Navier-Stokes system in the half space is investigated.  It is shown that if the initial
periodic perturbation around the viscous shock wave is small, then the solution time
asymptotically tends to a viscous shock wave with a shift partially determined by the
periodic oscillations. Moreover, the strength of {the} shock wave could be arbitrarily large. This result  essentially improves the previous work "{\it A. Matsumura, M. Mei, Convergence to travelling fronts of solutions of the p-system with viscosity in the presence of a boundary. Arch. Ration. Mech. Anal. 146 (1999), no. 1, 1-22.}"  where the strength of shock wave is sufficiently small and the  initial periodic oscillations  vanish.
\end{abstract}
\maketitle
\section{Introduction}

We consider a one-dimensional isentropic Navier-Stokes system for a general viscous gas, i.e.,
\begin{equation}\label{ns}
\left\{\begin{array}{ll}
v_t-u_x=0,&\\
u_t+p_x=(\mu(v)\frac{u_x}{v})_x,&
\end{array} \right.
\end{equation}
where $v(x,t) $ is the specific volume, $u(x,t)$   the fluid velocity and $p=a v^{-\gamma}$  is the pressure. Constant $a>0$, $\gamma> 1$  are adiabatic constants. $\mu(v)=\mu_{0}v^{-\alpha}$  is the viscosity coefficient with $\alpha\geq 0$.  Without loss of generality, we assume $\mu_0=1$ in what follows.


The system (\ref{ns}) is a basic system of hydrodynamic equations, it has a variety of wave phenomena, such as viscous shock waves and rarefaction waves.
So it is   important to study the stability of   the  viscous shock wave for  system (\ref{ns}). The stability of viscous shock wave for the Cauchy problem has been extensively studied in a large literature since the pioneer works of \cite{g1986,mn1985}, see the other interesting works \cite{fs1998, hm2009,HX, hlz2017,km1985,l1997,lz2009,lz2015,m,mn1994,sx1993,LWX2} as the shock wave is weak.

Physicists and engineers are more concerned with the stability of large amplitude shock (strong shock). However,  the stability of large amplitude shock (strong shock) is   challenging in  mathematics. There { have been }no research results of this area until the last few years. 

In 2010, Matsumura-Wang \cite{mw2010} proved  that the large amplitude shock wave is asymptotically stable by a clever weighted energy method as $\alpha >\frac{\gamma-1} {2} $. In 2016, Vasseur-Yao \cite{vy2016} successfully removed the condition $\alpha >\frac{\gamma-1} {2}$ by introducing a   new variable called  ``effective velocity''.
Recently, He-Huang  \cite{hh2020} extended the result of \cite{vy2016} to general pressure $p(v)$ and general viscosity $\mu(v)$, where $\mu(v)$ could be any positive smooth function.

On the other hand, it is also interesting to investigate the stability of viscous shock  {waves} for the initial-boundary value problem. In this paper, we considered an impermeable wall problem of (\ref{ns}) in the half space $x\ge 0$, i.e.,
\begin{equation}\label{im}
\left\{ \begin{array}{ll}
(v,u)(x,0)=(v_0,u_0)(x) \longrightarrow (v_++\zeta(x),u_++\varphi(x)), ~x\rightarrow +\infty, &\\
u(0,t)=0, ~~ t\in {R_+}, &
\end{array} \right.
\end{equation}
where $v_+>0,u_+<0$. And $ \left(\zeta, \varphi  \right) $ are periodic functions with period $ \pi >0 $ and satisfy
\begin{equation}\label{zero-ave}
\int_{0}^{\pi} \left(\zeta, \varphi\right)(x) dx =0.
\end{equation}

When the periodic functions $ \left(\zeta, \varphi  \right) $ vanish,  Matsumura-Mei \cite{mm1999} considered the impermeable wall problem   (\ref{ns}), (\ref{im})   in 1999.  {And recently, an interesting result by \cite{CLX} considering the multi-dimensional case of this problem.}

The impermeable wall means that the velocity at the boundary $x=0$ must be zero because there is no  flow across the boundary. They   showed   in \cite{mm1999} that  when $\alpha=0$ the solution of (\ref{ns}),(\ref{im})  tends to a  2-viscous shock wave connecting the left state $(v_-,0)$ and the right one $(v_+,u_+)$ provided that both the strength of shock and the initial perturbation are small and  the 2-viscous shock is initially  far away from the boundary, where $v_-$ is determined by the RH condition, i.e.,
\begin{eqnarray}\label{rh}
\left\{\begin{array}{ll}
-s_2(v_+-v_-)-(u_+-u_-)=0, & \\
-s_2(u_+-u_-)+(p(v_+)-p(v_-))=0. &
\end{array}
\right.
\end{eqnarray}
Moreover, we assume that $u_-=0$. The condition  {that} the strength of shock is small was removed from \cite{CL2021}. The condition {that} the shock is initially far away from the boundary  was removed from \cite{mn2004}. {How to} remove both these two conditions mentioned above at the same time is still open. Let us briefly recall the idea of \cite{mn2004}.

Since $u(0,t)=0$ at the boundary, we can exchange
the impermeable wall problem (\ref{ns}) and (\ref{im})
in the half space to the Cauchy problem in the whole space by defining
$( { {\tilde{v}}}(x,t),\tilde{u}(x,t))=(v(-x,t),-u(-x,t))$ as $x<0$ so that $(\tilde{v},\tilde{u})(x,t) $ still satisfies the system (\ref{ns}) in the whole space, i.e.,
\begin{equation}\label{1.6}
\left\{\begin{array}{ll}
\tilde{v}_t- \tilde{u}_x=0, &\\
\tilde{u}_t+p({\tilde{v}})_x=(\frac{{\tilde{u}}_x}{{\tilde{v}}^{\alpha+1}})_x,~~x\in {R}&
\end{array} \right.
\end{equation}
equipped with the initial data
\begin{equation}\label{1.7}
({\tilde{v}_{0}}, {\tilde{u}_{0}})(x)=:( {\tilde{v}}, {\tilde{u}})(x,0)=\left\{ \begin{array}{ll}
(v_{0}(-x),-u_{0}(-x)), & x\leq 0,\\
(v_{0}(x),u_{0}(x)), & x\geq 0,\\
\end{array} \right.
\end{equation}
satisfying
\begin{equation}\label{1.8}
(\tilde{v}_{0}, {\tilde{u}_{0}})(x)\rightarrow\left\{\begin{array}{ll}
(v_{+} ,u_{+} ),& (x\rightarrow  +\infty),\\
(v_{+} ,-u_{+} ), & (x\rightarrow  -\infty).
\end{array} \right.
\end{equation}

\begin{figure}[!h]
\begin{center}
\subfloat[Combination of the two shock waves]{\includegraphics[width=0.45\textwidth]{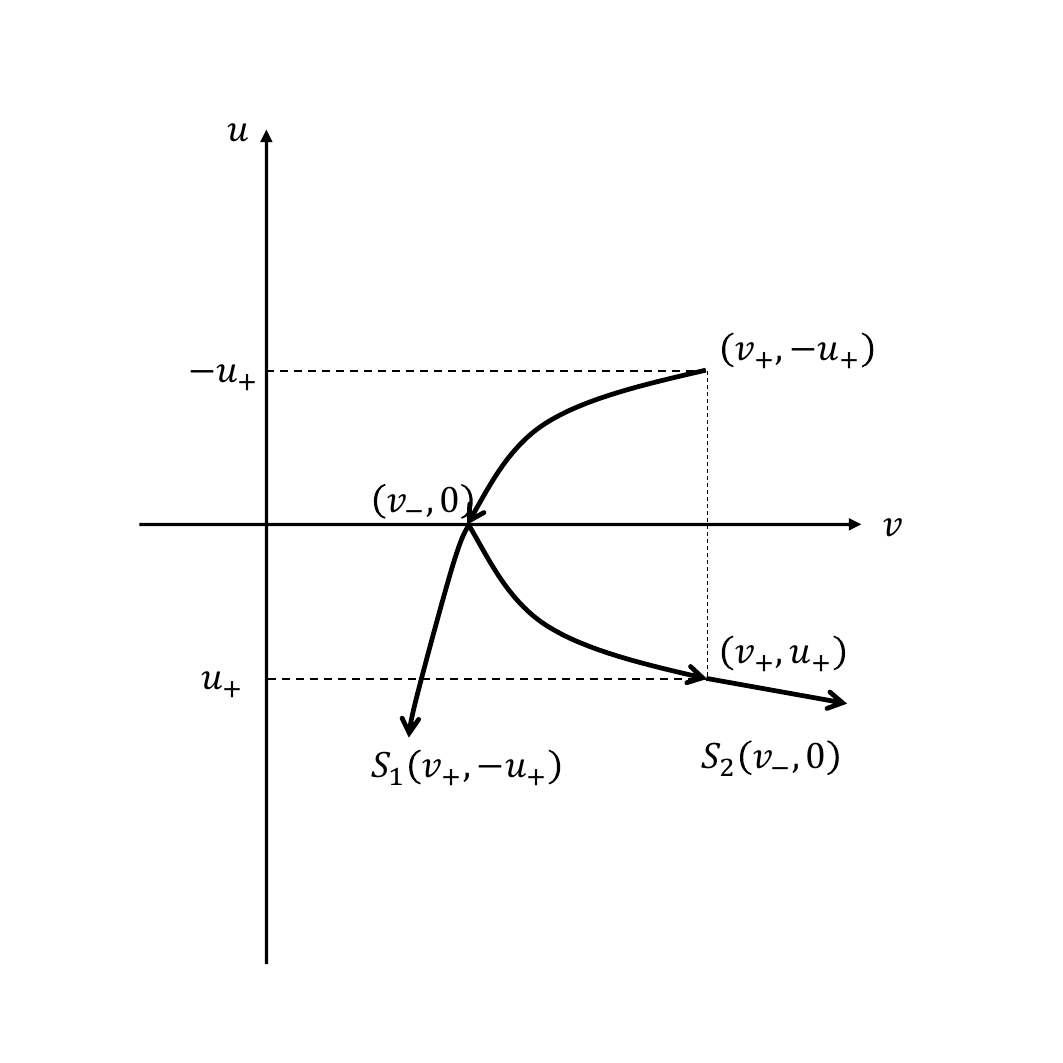}
\label{FIG1}}
\hfil
\subfloat[ The graphs of $V_{i} $ and $U_{i}  $ ,i=1,2 ]{\includegraphics[width=0.45\textwidth]{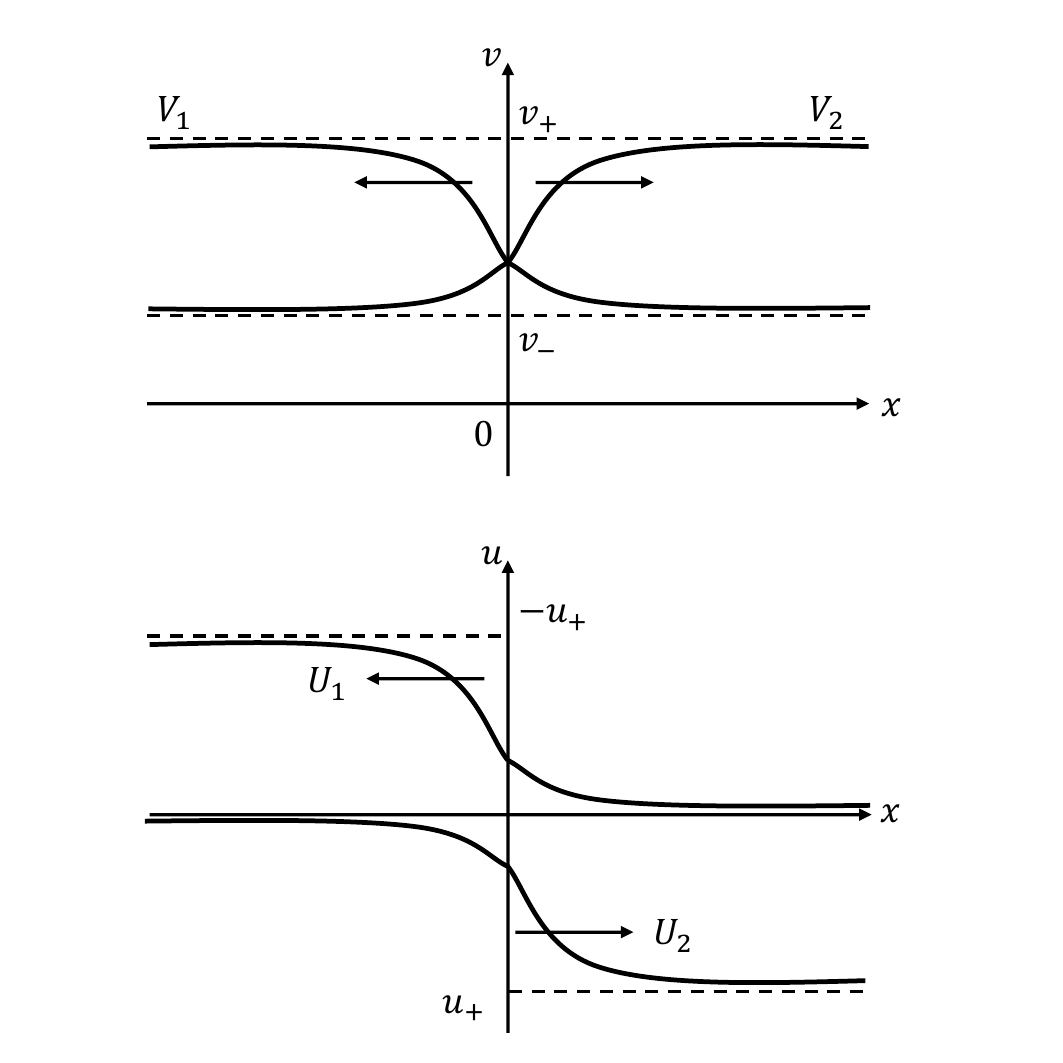}
\label{fig2}}
\end{center}
\end{figure}

It is obvious that the solution of the Cauchy problem (\ref{1.6})-(\ref{1.8}) confined in the half line $x>0$ is exactly the one of the impermeable wall problem (\ref{ns}),(\ref{im}). In view of the far field states at
$x=\pm \infty$ given by (\ref{1.8}), it is expected that the solution to (\ref{1.6})-(\ref{1.8}) asymptotically tends to a composite wave consisting of 1-viscous shock wave connecting $(v_+,-u_+)$ at the left and an intermediate  {state} $(v_\star,u_\star)$ at the right, and 2-viscous shock wave connecting $(v_\star,u_\star)$ at the left and $(v_+,u_+)$ at the right. Fortunately $(v_\star,u_\star)=(v_-,0)$ by the principle of RH condition and (\ref{1.8}), see
Fig (A), where $S_1(v_+,-u_+)$ means the 1-shock curve in the phase plane $(v,u)$ starting from the left state $(v_+,-u_+)$ and $S_2(v_-,0)$ means the 2-shock curve in the phase plane $(v,u)$ starting from the left state $(v_-,0)$.
The Figure B contains the graphs of the shock waves in the planes $(x,v)$ and
$(x,u)$.
The wall $x = 0$ can be regarded as a mirror and the 1-viscous shock is a mirror image of the 2-viscous shock in the plane $(x,v)$, and  the interaction between the 2-shock and the boundary $x=0$ for the impermeable wall problem (\ref{ns})-(\ref{im}) is replaced to consider the one between the 2-shock and its mirrored shock for the Cauchy problem (\ref{1.6})-(\ref{1.8}).


In this paper, we want to improve the work of \cite{CL2021} where $\zeta=\varphi=0$. Motivated by \cite{mn2004}, the extended   initial data in (\ref{1.7}) satisfies

\begin{equation}\label{1.9}
(\tilde{v}_{0}, {\tilde{u}_{0}})(x)\rightarrow\left\{\begin{array}{ll}
(v_{+}+\zeta(x),u_{+}+\varphi(x) ),& (x\rightarrow  +\infty),\\
(v_{+}+\zeta(-x),-u_{+}-\varphi(-x)), & (x\rightarrow  -\infty).
\end{array} \right.
\end{equation}

 We outline the strategy as follows. We apply the anti-derivative method to study the stability of the traveling wave solution  $(V_{2}^{S},U_{2}^{S})(x-s_2t)$, in which the anti-derivative of the perturbation ($\tilde{v}-V_{2}^{S}, \tilde{u}-U_{2}^{S}$), namely,  $(\phi,\psi)(x,t)=\int_{-\infty}^{x}(\tilde{v}-V_{2}^{S} , \tilde{u}-U_{2}^{S})(y,t)dy$, ``should'' belong to some Sobolev spaces like $H^2(\mathbb{R})$. However, the method above can not be applicable directly in this paper since ($\tilde{v}-U_{2}^{S}, \tilde{u}-U_{2}^{S}$) oscillates at the far field and hence does not belong to any $L^p$ space for $p\ge 1$. Motivated by \cite{XYY2019},  we introduce a suitable ansatz $(V,U)(x,t)$, which has the same oscillations as the solution $(\tilde{v},\tilde{u})(x,t)$ at the far field, so that $\int_{-\infty}^{x}(\tilde{v} -V, \tilde{u} -U)(x,t) dx$ belongs to some Sobolev spaces and the anti-derivative method is still available.

The rest of the paper will be arranged as follows. In Section \ref{section2}, a suitable ansatz is constructed and the main results are stated. In Section \ref{section3},  the stability problem is reformulated to a perturbation equation around the ansatz. In Section \ref{section4},   the a priori estimates are established.   In Section \ref{section5},   the main results   are  proved.    In Section \ref{Sec-shift-F}, some complementary proofs     are provided.

\

\noindent \textbf { Notation.}
The functional $\|\cdot\|_{L^p(\Omega)}$ defined by $\| f\|_{L^p(\Omega)} =  (\int_{\Omega}|f|^{p}(\xi ){ d\xi } )^{\frac{1}{p}}$.  When $\Omega=(-\infty,\infty)$, the symbol $\Omega$ is often omitted. As $p=2$, we denote for simplicity,
\begin{equation*}
\| f\| =  \left(\int_{ -\infty}^{ \infty}f^{2}(\xi ){ d\xi } \right)^{\frac{1}{2}}.
\end{equation*}
In addition, $H^m$ denotes the  $m$-th  order Sobolev space of functions defined by
 \begin{equation*}
\|f\|_{m} =  \left( \sum_{k=0}^{m}  \|\partial^{k}_{\xi}f\|^2 \right)^{\frac{1}{2}}.
\end{equation*}

\section{Preliminaries and the Main Theorem}\label{section2}
\subsection{Preliminaries}

As pointed out by \cite{mm1999,CL2021},  when perturbation functions $\zeta, \varphi$ vanish, the solution of the impermeable wall problem (\ref{ns})-(\ref{im}) is expected to tend toward the outgoing viscous shock $(V_2^{S}, U_2^{S})(\xi_2)$ satisfying
\begin{equation}\label{2.1}
\left\{ \begin{array}{ll}
{-s_2}(V_2^{S})'-(U_2^{S})'=0,&\\
{-s_2}(U^{S}_2)'+p(V^{S}_2)'=\left(\frac{(U^{S}_2)'}{(V^{S}_2)^{\alpha+1}}\right)',&\\
(V^{S}_{2},U^{S}_{2})(-\infty)=(v_-,0),\quad (V^{S}_{2},U^{S}_{2})(+\infty)=(v_+,u_+),&
\end{array} \right.
\end{equation}
where $'=d / d\xi_{2}$, $\xi_2=x-s_2t$, $s_2$ is the shock speed determined by the R-H condition (\ref{rh}) and $v_\pm>0,u_+<0$ are given constants. Using $(\ref{2.1})_1 $ and $(\ref{2.1})_2 $, it follows that
\begin{align}\label{2.2}
\begin{split}
&  {s}_2^{2}(V^{S}_2)'+p(V^{S}_2)'= \left(\frac{-s_{2}(V^{S}_2)'}{(V^{S}_2)^{\alpha+1}}\right)'.
\end{split}&
\end{align}
Integrate (\ref{2.2})  over $(-\infty,\xi_{2})$. one has
\begin{align*}
\begin{split}
&\frac{s_2 (V^{S}_2)'}{(V^{S}_2)^{\alpha+1}}=-{s}_2^{2}(V^{S}_2)-p(V^{S}_2) -b:= h(V^{S}_2), \quad  V^{S}_2(\pm\infty)=v_\pm,
\end{split}&
\end{align*}
\begin{align}
\begin{split}
&U^{S}_2=-s_2(V^{S}_2-v_-)=-s_2(V^{S}_2-v_+)+u_+,
\end{split}&
\end{align}
where $b=-s_2^2 v_--p(v_-)=-s_2^2 v_+-p(v_+)$. For abbreviation, we denote $s_2$ by $s$. We have the following lemma.
\begin{lma}[\cite{mm1999}]\label{lemma2.1}
There exists a unique viscous shock  $(V^{S}_{2}, U^{S}_{2})(\xi_2)$ up to a shift  satisfying
\begin{eqnarray*}
 0<v_{-}<V^{S}_2(\xi_2)<v_{+}, \quad
h(V^{S}_{2})>0, \quad (U^{S}_2)'<0,
\end{eqnarray*}
\begin{eqnarray}\label{2.4}
\left|V^{S}_{2}(\xi)-v_{\pm}\right|=O(1)   \theta  e^{-c_{\pm}|\xi_{2}|},
\end{eqnarray}
as $\xi_{2} \rightarrow \pm \infty,$ where    $\theta=\left|v_{+}-v_{-}\right|$, $c_\pm=\frac{v_{\pm}^{\alpha+1}}{s} |p'(v_\pm)+s^2|$,
$s=\frac{-u_+}{v_+-v_-}. $
\end{lma}

The initial data are  assumed to satisfied
\begin{align}\label{2.16}
\begin{split}
&v_0(x) -\zeta(x) - V_{2}^S(x-\beta_{1})\in L^1\cap H^1(R_+),\\
&u_0(x)-\varphi(x) -U_{2}^S(x-\beta_{1})  \in L^1\cap H^1(R_+),
\end{split}&
\end{align}

 and
\begin{equation}
u_0(0)=0
\end{equation}
as compatibility condition, where   $\beta_{1}>0$ is a constant.
Set
\begin{eqnarray*}
(A_{0},B_0)(x):= -\int_{x}^{\infty} (v_{0}(y)-\zeta(y)-V_{2}^S(y-\beta_{1}), u_{0}(y)- \varphi(y)-U_{2}^S(y-\beta_{1})  )dy.
\end{eqnarray*}
We further assume that
\begin{eqnarray}\label{2.18}
( A_{0} ,B_0 ) \in L^2 (R_+).
\end{eqnarray}

Borrowing  from the idea of \cite{mn2004}, we construct a composite wave.
By \cite{mn2004}, the mirrored shock $(V^{S}_1,U^{S}_1)(\xi_1), \xi_1=x-s_1t, s_1=-s$, satisfies
\begin{equation}\label{2.3}
\left\{\begin{array}{ll}
{s}(V^{S}_1)'-{U^{S}_1}'=0,&\\
{s}(U^{S}_1)'+p(V^{S}_1)'=(\frac{(U^{S}_1)'}{(V^{S}_1)^{\alpha+1}})',&\\
(V^{S}_{1},U^{S}_{1})(-\infty)=(v_+,-u_+),\quad (V^{S}_{1},U^{S}_{1})(+\infty)=(v_-,0).&
\end{array}\right.
\end{equation}
Thanks \cite{mm1999}, one has
\begin{equation}\label{2.5}
V^{S}_1(\xi)=V^{S}_2(-\xi), \quad U^{S}_1(\xi)=-U^{S}_2(-\xi), ~~\forall \xi\in {R}.
\end{equation}
The composite wave by two viscous shock weaves is defined as
\begin{align}\label{{2.6}}
\begin{split}
\tilde{V}(x,t;\beta):=& V_1^{S}(x+st+\beta)+V^{S}_2(x-st-\beta )- v_-,\\
\tilde{U}(x,t;\beta):=& U_1^{S}(x+st+\beta)+U^{S}_2(x-st-\beta ),
\end{split}
\end{align}
where $\beta$ is a constant. Motivated by \cite{HXY2022,HXXY,LWX,HY1shock}, we need two periodic solutions to (\ref{ns}) to establish the ansatz. Some properties of the solution are listed.
\begin{lma}\cite{HY1shock}\label{Lem-periodic}
Assume that $ (v_0,u_0 )(x)\in H^k (0,\pi)  $ with $ k\geq 2 $ is periodic with period $ \pi>0 $ and average $ (\bar{v},\bar{u} ). $ Then there exists $ \e_0>0 $ such that if
\begin{equation*}
\e_{1}:=\|{(v_0,u_0)-( {\bar{v}}, {\bar{u}})}\|_{H^k(0,\pi)} \leq \e_0,
\end{equation*}
there exits a unique periodic solution $$ (v,u)(x,t) \in C\big(0,+\infty;H^k(0,\pi) \big) $$ to (\ref{ns}) with the initial data $ (v,u )(x,0) = (v_0,u_0 )(x), $ which has the average $ ( {\bar{v}}, {\bar{u}} ), $ and satisfies
\begin{equation*}
\|{(v,u )-( {\bar{v}}, {\bar{u}} )}\|_{H^k( 0,\pi) }(t) \leq C \e_{1} e^{-2\sigma_{0} t}, \quad t\geq 0,
\end{equation*}
where the constants $ C>0 $ and $ \sigma_{0}>0 $ are independent of $ \e_{1} $ and $ t. $
\end{lma}
\subsection{Ansatz}
In order to make the anti-derivative method is available, we choose a suitable pair of  \textit{ansatz} $(\vt, \ut )$ such that $\lim_{x \rightarrow \pm\infty}(v-V, u-U)(x,t)=(0,0)$ for any $t\geq0$.
Motivated by \cite{Xin2019}, we define that   the periodic solutions $ \left(v_{l,r}, u_{l,r} \right) $ of (\ref{ns}) as $x \rightarrow \mp \infty $ for all $ t\geq 0, $ which have the periodic initial data:
\begin{align*}
\begin{split}
&\left(v_r, u_r \right)(x,0)=(v_+,u_+)+\left(\zeta,\varphi \right)(x),\\
&\left(v_l, u_l \right)(x,0)=(v_+,-u_+)+\left(\zeta,-\varphi \right)(-x).
\end{split}
\end{align*}
For the viscous shocks $ \left(V^{S}_1, U^{S}_1\right) $ and $ \left(V^{S}_2,U^{S}_2\right), $ define
\begin{align}\label{2.7}
\begin{split}
g_1(x)&:=\frac{V^{S}_1(x)-v_+}{v_--v_+}=\frac{U^{S}_1(x)+u_+}{u_+},\\
g_2(x)&:=\frac{V^{S}_2(x)-v_-}{v_+-v_-}=\frac{U^{S}_2(x)}{u_+},
\end{split}
\end{align}
where we have used the R-H condition (\ref{rh}). It is straightforward to check that $ 0\leq g_i(x)\leq 1 $ and $ g_i'(x)>0 $  for $ i=1,2. $ With functions $v_{l,r}, u_{l,r}, g_{1,2}$ in hand,  we are ready to construct the ansatz. Let $\X(t), \Y(t)$ are two $ C^1$ curves on $[0,+\infty)$ which will be determined later.
Set
\begin{equation*}
\begin{aligned}
\vt(x,t) :=~ & v_l(x,t) \left[ 1-   g_1 (x+st+{\X}) \right] + v_- \left[   g_1 (x+st+{\X}) -   g_2 (x-st-{\X}  ) \right] \\
& + v_r(x,t)   g_2 (x-st-{\X} ), \\
\ut(x,t) :=~ & u_l(x,t) \left[ 1-   g_1 (x+st+{\Y}) \right] + u_r(x,t)g_2 (x-st-{\Y} ).
\end{aligned}
\end{equation*}
Plugging the ansatz $ (\vt, \ut) $ into (\ref{ns}), we have
\begin{equation}\label{2.8}
\begin{cases}
 \vt_t-\ut_x =(F_{1,1} + F_{1,2} + \X' F_{1,3})_x, & \\
 \ut_t+p(\vt)_x - \mu\big(\frac{\partial_x \ut}{\vt}\big)_x =(F_{2,1} + F_{2,2}+ \Y' F_{2,3})_x, &
\end{cases}
\end{equation}
where

\begin{equation*}
F_{i,j}= \int_{-\infty}^x  f_{i,j} (y,t)dy; i=1,2;\quad j=2,3,
\end{equation*}
\begin{equation*}
\begin{cases}
F_{1,1}=u_l\left[g_1(x+st+{\Y} )-g_1(x+st+{\X}) \right] \\
\qquad\quad-u_r\left[g_2(x-st-{\Y} )-g_2(x-st-{\X})\right] , \\
f_{1,2}=\left[-s(v_l-v_+)+(u_l+u_+)\right]g_1'(x+st+{\X})\\
\qquad\quad-\left[s(v_r-v_+)+(u_r-u_+)\right]g_2'(x-st-{\X}), \\
f_{1,3}=(v_--v_l)g_1'(x+st+{\X})\\
 \qquad\quad+(v_--v_r)g_2'(x-st-{\X}),
\end{cases}
\end{equation*}
and
\begin{equation*}
\begin{cases}
F_{2,1} = p(\vt ) - p(v_l ) \left[ 1-   g_1(x+st+{\Y} )  \right]  - p(v_r )  g_2 (x-st-{\Y} ) \\
\qquad \quad -\big[\frac{  \ut_x }{\vt^{\alpha+1}} - \frac{  u_{lx}}{v_l^{\alpha+1}} ( 1- g_1(x+st+{\Y} )  ) - \frac{  u_{rx}}{v_r^{\alpha+1}}  g_2 (x-st-{\Y} )\big], \\
f_{2,2} = \big[ -s u_l - p(v_l ) +   \frac{ u_{lx}}{v_l^{\alpha+1}}  \big] g'_1(x+st+{\Y} )  )  \\
\qquad\quad - \big[ s u_r - p(v_r ) +   \frac{ u_{rx}}{v_r^{\alpha+1}} \big]  g_2'(x-st-{\Y} ), \\
f_{2,3} = -u_l g'_1(x+st+{\Y} )  ) - u_r  g_2'(x-st-{\Y} ).
\end{cases}
\end{equation*}

\subsection{Location of The Shift ${\X}(t)$ and ${\Y}(t)$.}
To apply the anti-derivative  method which is always used to study the stability of viscous shock, introduced in \cite{O}, we expect that
\begin{align*}
\begin{split}
0=&\int_{-\infty}^{\infty}\left(
\begin{array}{cccc}
 {\tilde{v}} (x,t) -    V(x,t )  \\
\tilde{u} (x,t) -    U(x,t )
\end{array}
\right ), \quad \forall t\geq0.
\end{split}&
\end{align*}
When $t=0$,  the shifts $\X(0)$ and $\Y(0)$  should satisfy
\begin{align}\label{2.9}
\begin{split}
0=&\int_{-\infty}^{\infty}\left(
\begin{array}{cccc}
\tilde{v}_0(x)-V(x,0)\\
\tilde{u}_0(x)-U(x,0)
\end{array}
\right )
dx:=\left(\begin{array}{cccc}
I_1(\X(0))  \\
I_2(\Y(0))
\end{array}
\right).
\end{split}&
\end{align}

Our next task is to  show    $\X(t), \Y(t)$ when $ t>0$. To make the system (\ref{2.8}) as a conservative form, the curves $ \X (t)$ and $ \Y (t)$ should satisfy
\begin{equation}\label{ode-shift}
\begin{aligned}
\X'(t) = -\lim_{x\rightarrow\infty}\frac{F_{1,2}(x,t)}{F_{1,3}(x,t)}, ~
\Y'(t) = -\lim_{x\rightarrow\infty}\frac{F_{2,2}(x,t)}{F_{2,3}(x,t)},
\end{aligned}
\end{equation}
 With the aid of  (\ref{Lem-periodic}), we know $F_{1,3}\neq 0 ,F_{2,3}\neq 0 $, provided that the initial periodic perturbations $ \left(\zeta, \varphi \right)  $ are small.
 Due to (\ref{1.7}) and (\ref{2.5}), $ \tilde{u}_0(x), U(x,0 )$ are odd functions and $ {\tilde{v}}_0(x), V(x,0 )$ are even functions, thus $I_{02}=0$, i.e, we can choose any $\Y_{0}$ to guarantee that $I_2( \Y_{0} )=0$. For $I_2( \X_{0}  )=0$, using (\ref{2.4}) and (\ref{2.5}),  one gets that
\begin{align}
\begin{split}
I_1(\omega)=2&\int_{0}^{\infty}[v_{0}(x)-\tilde{q}(x)-V^{S}_2 (x-\omega)]+[v_--V^{S}_1 (x+\omega)]dx\\
=2& \int_{0}^{\infty}[v_{0}(x)-\tilde{q}(x)-V^{S}_2 (x-\omega)]dx-2\int_{0}^{\infty}\frac{1}{s}U^{S}_1(x+\omega)dx\\
=2&\int_{0}^{\infty}[{v}_{0}(x)-\tilde{q}(x)-V^{S}_2 (x-\omega)]dx+2\int_{0}^{\infty}\frac{1}{s}U^{S}_2(-x-\omega)dx\\
=2&\int_{0}^{\infty}[{v}_{0}(x)-\tilde{q}(x)-V^{S}_2 (x-\omega)]dx+2\int_{0}^{\infty}U^{S}_2( -s t-\omega)dt.
\end{split}&
\end{align}
where
\begin{align}
\begin{split}
\tilde{q}(x)=&\zeta(-x)\left[1-g_1(x+\omega)\right]+\zeta(x)\left[g_2(x-\omega)\right].
\end{split}
\end{align}

By directly calculate, we have $I_1(\infty)=\infty$, $I_1(-\infty)=-\infty$.
\begin{align}
\begin{split}
I_1'(\omega)&=2 \left\{(v_+-v_-)-\int_{0}^{\infty}   [\zeta(-x)g_{1}'(x+\omega)+\zeta(x)g_{2}'(x-\omega)]dx\right\}\\
&\geq 2 \left\{(v_+-v_-)-\|\zeta\|_{L_{\infty}}\int_{0}^{\infty}  [g_{1}'(x+\omega)+g_{2}'(x-\omega)]dx\right\}\\
&\geq 2 \left\{(v_+-v_-)-2\varepsilon\right\}.
\end{split}&
\end{align}
Moreover, choosing $\varepsilon$ suitable small, we have $3(v_+-v_- )\geq I_1'(\omega) \geq  v_+-v_- >0$
Thus there exists a unique constant $\X_{0}$  such that $I_1(\X_{0})=0$. Moreover, using $I_1(\X_{0})=I_1(\beta_{1})+\int_{\beta_{1}}^{\X_{0}}I_1'(s) ds$, the constant $\X_{0}$ is between $\frac{1}{2}\tilde{M}+\beta_{1}$  and    $\frac{3}{2}\tilde{M}+\beta_{1}$,  where

\begin{align}\label{0507-1}
\begin{split}
 \tilde{M}= & \frac{1}{v_--v_+}\left(\int_{0}^{\infty}[{v}_{0}(x)-\tilde{q}(x)-V^{S}_2 (x-\beta_{1} )]dx+ \int_{0}^{\infty}U^{S}_2( -s t-\beta_{1} )dt\right) \\
= & \frac{1}{v_--v_+}\left(\int_{0}^{\infty}[{v}_{0}(x)-\zeta(x)-V^{S}_2 (x-\beta_{1} )]dx+ \int_{0}^{\infty}U^{S}_2( -s t-\beta_{1} )dt\right) \\
 &-\frac{1}{v_--v_+}\left(\int_{0}^{\infty}\left[\zeta(x)[g_2(x-\beta_{1})-1\right]+\zeta(-x)\left[1-g_1(x+\beta_{1})\right]dx\right)\\
  \leq&  \frac{1}{v_--v_+}\left(\int_{0}^{\infty}[{v}_{0}(x)-\zeta(x)-V^{S}_2 (x-\beta_{1} )]dx+ \int_{0}^{\infty}U^{S}_2( -s t-\beta_{1} )dt\right) + C\varepsilon,
 \end{split}
\end{align}
where we have used the following inequality

\begin{align}\label{0507-2}
\begin{split}
&  \int_{0}^{\infty}\left[\zeta(x)[g_2(x-\beta_{1})-1\right]+\zeta(-x)\left[1-g_1(x+\beta_{1})\right]dx \\
=&\frac{-1}{(v_--v_+) } \int_{0}^{\infty}\zeta(x)\left[V^{S}_2(x-\beta_{1})-v_+\right]+\zeta(-x)\left[V^{S}_1(x+\beta_{1})-v_-\right]dx \\  \leq&  C\|\zeta(x)\|_{L_{\infty}}\int_{0}^{\infty}\left|V^{S}_2(x-\beta_{1})-v_+\right|+\left|V^{S}_1(x+\beta_{1})-v_-\right|dx \leq C\varepsilon.
 \end{split}
\end{align}
By (\ref{2.16}), we know $\tilde{M}$ exists. Thus we can obtain the   curves $ \X(t) $ and $ \Y (t) $. More precisely, it holds that
\begin{lma}\label{Lem-shift}
Assume that (\ref{zero-ave}), (\ref{rh}) hold. Then there exists an $ \e_0>0 $ such that if
\begin{equation*}
\|{\zeta}\|_{H^2(0,\pi)} +\|\varphi \|_{H^2(0,\pi)}<\e < \e_0,
\end{equation*}
there exists a constant pair $ (\X, \Y)(0)$  satisfying  (\ref{2.9}) where  $ \X(0)$ is uniquely determined and $ \Y (0)$ can take any constant. Moreover, there exists a unique solution $ (\X, \Y )(t) \in C^1(0,+\infty) $ to the system   (\ref{ode-shift}) with the fixed initial data $ (\X, \Y)(0)=(\X_0, \Y_0$)  satisfying
\begin{equation*}
|{\left(\X',\Y'   \right)(t)}| + |{\left(\X,\Y\right)(t)-\left(\X_\infty, \Y_\infty \right)}| \leq C\e e^{-2\sigma_{0} t}, \qquad t\geq 0.
\end{equation*}	
Moreover, the corresponding constant locations $ \X_\infty, \Y_\infty $ as follows,
\begin{align}\label{X-inf}
\X_\infty =~& \X_0 +\frac{1}{2(v_+-v_-)\pi}\left\{   \int_{0}^{\pi} \int_{0}^{x}\zeta (y)-\zeta (-y)dy dx  \right. ,\\
		& \left.+2\pi \int_{0}^{+\infty} \left[ \zeta(x) \left(1 -g_ 2(x+\X_0)\right) - \zeta(x) \left( 1-g_1(x-\X_0)\right) \right] dx \right\},\notag
\end{align}
and
\begin{align}\label{Y-inf}
\Y_\infty =~ & \Y_0 +	
\frac{1}{2 u_+ \pi} \left\{ \int_{0}^{\pi} \int_{0}^{x} \varphi(y)+\varphi(-y)dydx\right. \\
&    - \int_{0}^{+\infty} \int_{0}^{\pi} \left[p(v_l )-p( {v}_r )\right] dx dt   \notag\\
& \left. + \int_{0}^{\pi} g(v_++\zeta(-x))-g(v_++\zeta(x))dx  \right\}, \notag
\end{align}

where $g(v)=\frac{1}{\alpha}v^{-\alpha},$ if $ \alpha\neq0$; $ g(v)= -\ln v$, if  $\alpha\neq0$.
\end{lma}

Since the proof of Lemma \ref{Lem-shift} is similar to that in \cite{Xin2019,HY1shock}, we put it in   section \ref{Sec-shift-F}.
\subsection{the Main Result}


\color{black}

  We define
\begin{align*}
\begin{split}
&\phi_0(x)=-\int^\infty_x \widetilde{v}_0(y)-V(y,0 )dy,\\
&\psi_0(x)=-\int^\infty_x \widetilde{u}_0(y)-U(y,0 )dy.
\end{split}
\end{align*}
In view of (\ref{2.9}),  we  further assume that
\begin{eqnarray}\label{2.13}
 (\phi_{0},\psi_0 )  \in H^2 (R ) .
\end{eqnarray}
Using the arbitrariness of $\Y_0$, one  can find a suitable $\Y_0 $, such that $\X_\infty=\Y_\infty$. From now on, we denote $\beta:=\X_\infty=\Y_\infty, \tilde{V}(x,t;\beta)=\tilde{V}(x,t), \tilde{U}(x,t;\beta)=\tilde{U}(x,t)$ for simple.
\begin{lma}\label{ls}
Suppose that  (\ref{2.13}) holds, there exists a positive constant $\delta_{1}$ such that if $$\|\phi_{0}\|_{2}+\|\psi_{0}\|_{2 }+\beta_{1}^{-1}+\varepsilon\leq\delta_{1}  ,  $$ then the Cauchy problem
(\ref{1.6}),(\ref{1.9}) has a unique global solution $(\widetilde{v},\widetilde{u}) (x,t) $ satisfying
\begin{align}\label{2.14}
\begin{split}
&{\widetilde{v}(x,t)}-V(x,t )\in C^0([0,+\infty);H^1) \cap  L^2([0,+\infty);H^2), \\
&{\widetilde{u}(x,t)}-U(x,t )\in C^0([0,+\infty);H^1) \cap  L^2([0,+\infty);H^1), \\
\end{split}
\end{align}
and
\begin{align}\label{2.15}
\begin{split}
&\sup_{x\in {R} }| {\widetilde{v}(x,t)}-V(x,t ) |     \rightarrow 0,  \text{   as } t\rightarrow +\infty,\\
&\sup_{x\in {R} }| {\widetilde{u}(x,t)}-U(x,t ) |     \rightarrow 0,   \text{   as } t\rightarrow +\infty.
\end{split}
\end{align}
\end{lma}

Now, we turn to the original initial-value problem.
Our main theorem is:
\begin{thm}\label{theorem}
 For any given constants $u_+<0$ and  $v_+>0$, if  (\ref{2.16})-(\ref{2.18})  hold. There exists a positive constant $\delta_{2}$ such that if $$\|A_{0}\|_{H^{2}({R}_{+})}+\|B_{0}\|_{H^{2}({R}_{+})}+\beta_{1}^{-1}+\varepsilon\leq\delta_{2},  $$then the IBVP (\ref{ns}), (\ref{im}) has a unique global solution $(v,u)(x,t)$, satisfying
\begin{align*}
\begin{split}
&\sup_{x\in {R}_{+}} |{(v,u)(x,t)}-(V^{S}_{2}, U^{S}_{2})(x-st -\beta) |\rightarrow 0, \text{as }   t\rightarrow +\infty,
\end{split}
\end{align*}
where $\beta $ is determined by (\ref{X-inf}).
\end{thm}
\section{Reformulation of the Original Problem}\label{section3}
Set
\begin{align*}
\begin{split}
\phi&( x,t):=\int^x_ {-\infty}{(\tilde{v}-V)}(y,t)  {d}y,  \\
\psi&( x,t):=\int^x_{-\infty} {(\tilde{u}-U)}(y,t)  {d}y.
\end{split}
\end{align*}
Thus $(\tilde{v},\tilde{u})(x,t)$ satisfy
\begin{align*}
\tilde{v}&( x,t)=\phi_x (x,t)+  V(x,t ), \nonumber\\
\tilde{u}&( x,t)=\psi_x (x,t)+  U(x,t ).
\end{align*}
From $(\ref{2.8}) $, we know the ansazt $(V,U)$ satisfies
\begin{equation}\label{3.1}
\left\{ \begin{array}{ll}
{V}_{t}-{U}_{x}=-F_{1x},&\\
U_{t}+p(V)_{x}-\left(\frac{U_{x}}{V^{\alpha+1}}\right)_{x}=-F_{2x },&\\
(V,U)(\pm \infty,t)=(v_{+}, \pm u_{+}),&
\end{array} \right.
\end{equation}
where
\begin{equation}\label{3.2}
\begin{aligned}
	F_1(x,t) & := -F_{1,1}(x,t)-F_{1,2}(x,t)-\X'(t)  F_{1,3}(x,t), \\
	F_2(x,t) & := -F_{2,1}(x,t)-F_{2,2}(x,t)-\Y'(t)  F_{2,3}(x,t).
\end{aligned}
\end{equation}
Motivated by  \cite{mm1999} and \cite{C2021}, with the help of (\ref{3.1}) and (\ref{ns}), it follows that
\begin{equation}\label{3.3}
\left\{ \begin{array}{ll}
\phi_{t}-\psi_{x} =F_1, &\\
\psi_{t}-f(V,U_x)\phi_{x}-\frac{\psi_{x x}}{V^{\alpha+1}} =F_2+J.&
\end{array} \right.
\end{equation}
The initial condition satisfies
\begin{align}\label{3.4}
\left(\phi_{0},\psi_{0}\right)&(x) \in {H}^{2}, \quad x \in {R},
\end{align}
where
\begin{align}\label{3.5}
f(V,U_x)= -p^{\prime}(V)    - (\alpha+1)\frac{ U_{x}}{V^{\alpha+2}},
\end{align}
\begin{align}\label{3.6}
J&= \frac{{u}_{x}}{{v}^{\alpha+1}}-\frac{U_{x}}{V^{\alpha+1}} - \frac{\psi_{xx}}{V^{\alpha+1}}+(\alpha+1)\frac{U_x \phi_x }{V^{\alpha+2}}-\left[p({v})-p(V)-p'(V)\phi_x\right].
\end{align}
\begin{lma}\label{Lem-F}
	Under the assumptions of Theorem \ref{theorem}, the anti-derivative variables (\ref{3.2}) exist and satisfy that
\begin{equation*}
\|{ F_1 }\|_2    \leq C \e e^{- \sigma_{0} t}, 	   \| F_2  \|_1  \leq C \e e^{- \sigma_{0} t} + C e^{-c_-\beta_{1}}e^{-s c_- t}.
\end{equation*}

\end{lma}
The proof is based on Lemma  \ref{Lem-periodic},   Lemma \ref{Lem-shift}  and Lemma \ref{qusiba} and we place it in section \ref{Sec-shift-F} for brevity.
\color{black}
\

We will seek the solution in the functional space ${X}_{\delta}(0,T)$ for any $0\leq T < +\infty $,
\begin{align*}
\begin{split}
{X}_{\delta}(0,T):=&\left\{ (\phi,\psi)\in C ([0,T];{H}^2)|\phi_{x} \in  {L}^2(0,T;{H}^1) ,\psi_{x} \in  {L}^2(0,T;{H}^2)\right. \\
&  \sup_{0\leq t\leq T}\|(\phi, \psi)(t)\|_2\leq\delta    \},
\end{split}&
\end{align*}
 where ${ \delta} \ll 1$ is small.

\begin{Remark}
The function space is well defined because the Dirac function will not appear in $\phi , \phi_{x},\phi_{xx},\psi, \psi_{x},\psi_{xx},\psi_{xxx}$, which can  be guaranteed by $u(0)=0$.
\end{Remark}

\begin{pro}\label{prposition3.1} (A priori estimate)
For some time $T>0$, if  $(\phi,\psi) \in {X}_{\delta}(0,T)$ is the solution of  (\ref{3.3}), (\ref{3.4}). Then there exists a positive constant $\delta_0  $ independent of  $T$, such that if
$$ \sup_{0\leq t\leq T}\|(\phi, \psi)(t)\|_{{2}} \leq \delta\leq \delta_{0},$$   for $t \in [0,T]$,
then
\begin{eqnarray*}
\|(\phi, \psi)(t)\|_{{2}}^{2}+    \int_{0}^t     (  \|  \phi_{x}(t)  \|^2_{1}  +   \|\psi_{x}(t)\|_{2}^2   )    {d}t       \leq C_{0} ( \|(\phi_{0},\psi_{0})\|_{{2}}^2+   e^{-c_-\beta_{1}}+\varepsilon),
\end{eqnarray*}
where $C_{0}  >1$   ia a  constant  independent of $T$.
\end{pro}

Once  Proposition \ref{prposition3.1} is obtained,   the local solution $(\phi,\psi)$   can be extend to $T =+\infty. $ See the following lemma.

\begin{lma}\label{lemma3.1}
  If $(\phi_{0},\psi_{0})\in {H}^2$, there exists a positive constant $\delta_{1}=\frac{\delta_{0}}{\sqrt{C_{0}}}$, such that if
$$\|(\phi_{0},\psi_{0})\|_{{2}}^2+   e^{-c_-\beta_{1}} +\varepsilon \leq \delta_{1}^{2},$$ then  the initial value problem  (\ref{3.3}), (\ref{3.4}) has a unique global solution
$(\phi,\psi)\in {X}_{\delta_{0}}(0,\infty)$  satisfying
\begin{align*}
\sup_{t\geq0}\|(\phi, \psi)(t)\|_{{2}}^{2}+    \int_{0}^\infty   (  \|  \phi_{x}(t)  \|^2_{1}  +   \|\psi_{x}(t)\|_{2}^2   )    {d}t       \leq C_{0}  ( \|(\phi_{0},\psi_{0})\|_{{2}}^2+   e^{-c_-\beta_{1}}+\varepsilon).
\end{align*}
\end{lma}

\section{A Priori Estimate}\label{section4}
For some $T>0$,   the problem $(\ref{3.3}), (\ref{3.4})$  is assumed that has a solution $(\phi,\psi)\in {X}_{\delta}(0,T) $ in this section.
\begin{eqnarray}\label{4.1}
\sup_{0\leq t\leq T}\|(\phi, \psi)(t)\|_{2}\leq \delta.
\end{eqnarray}
The Sobolev inequality  gives that $\frac{1}{2}v_{-}\leq v \leq \frac{3}{2}
v_{+}$, and
\begin{align*}
\sup_{0\leq t\leq T}  \{  \|(\phi, \psi)(t)\|_{{L}^{\infty}}    +   \|(\phi_{x}, \psi_{x})(t)\|_{{L}^{\infty}}   \}  \leq  {\delta} .
\end{align*}

 Motivated by  \cite{vy2016}, we introduce  the new effective velocity $\widetilde{h}=\widetilde{u}-\widetilde{v}^{-(\alpha+1)}\widetilde{v}_{x}$. It holds that
\begin{equation}\label{4.2}
\left\{ \begin{array}{ll}
\widetilde{v}_t-\widetilde{h}_x=(\frac{\widetilde{v}_{x}}{\widetilde{v}^{\alpha+1}})_{x},&\\
\widetilde{h}_t+\widetilde{p}_x=0.&
\end{array} \right.
\end{equation}
Similarly, we define $H=U-V^{-(\alpha+1)}V_{x}$, then (\ref{3.1}) becomes
\begin{equation}\label{4.3}
\left\{ \begin{array}{ll}
V_{t}-H_{x}=\left(\frac{V_{x}}{V^{\alpha+1}}\right)_{x}-F_{1,x},&\\
H_{t}+p(V)_{x}=-F_{2,x}.&
\end{array} \right.
\end{equation}\\
We define
\begin{eqnarray}\label{4.4}
 \int_{-\infty}^x(\widetilde{h}-H){d x}=\Psi.
\end{eqnarray}

Substitute (\ref{4.3}) from (\ref{4.2}) and integrate the resulting system with respect to $x$.  Using  (\ref{4.4}), we have
\begin{equation}\label{4.5}
\left\{ \begin{array}{ll}
\phi_t- \Psi_x-\frac{\phi_{xx}}{V^{\alpha+1}} =G+{F}_{1},&\\
\Psi_t+p'(V)\phi_x=-p(\tilde{v}|V)+F_{2}-\frac{F_{1x}}{V^{\alpha+1}},&
\end{array} \right.
\end{equation}
where
\begin{eqnarray*}
G=\frac{\widetilde{v}_{x}}{\widetilde{v}^{\alpha+1}}-\frac{V_{x}}{V^{\alpha+1}} - \frac{\phi_{xx}}{V^{\alpha+1}} ,\quad p(\widetilde{v}|V)=\left(p(\widetilde{v})-p(V)\right)-p'(V)\phi_x.
\end{eqnarray*}
Now we give some lemmas that are useful in energy estimate.
\begin{lma} (\cite{hh2020,mm1999})
Under the assumption of (\ref{4.1}), we have
\begin{align}\label{4.6}
\begin{split}
&p(\widetilde{v}|V)\leq C \phi_{x}^{2},  \quad |p(\widetilde{v}|V)_{x}|\leq C (|\phi_{xx}\phi_{x}|+|V_x|\phi_{x}^{2}),\\
&|G|\leq C (|\phi_{xx}\phi_{x}|+|V_x|\phi_{x}),\\
\end{split}
\end{align}
and
\begin{align}\label{4.7}
\begin{split}
&|J|\leq C( \phi_{x} ^{2}+|\phi_{x}\psi_{xx}|),\\
&|J_{x}|\leq C(\phi_{x}^{2}+|\phi_{x} \phi_{x x}|+|\psi_{x x} \phi_{x x}|+|\psi_{x x x} \phi_{x}|+|\phi_{x} \psi_{x x} |).
\end{split}
\end{align}
\end{lma}

\begin{lma}\label{ddterm-1}
The error terms
\begin{align}\label{4.8}
\begin{split}
q(x,t):=\vt(x,t)-\tilde{V}(x,t); \quad z(x,t):=U(x,t)-\tilde{U}(x,t),
\end{split}
\end{align}
satisfy
\begin{equation*}
\| (z,q)(.,t)\|_{2}  \leq C\e e^{-2\sigma_{0} t}.
\end{equation*}
\end{lma}
\begin{proof}
By direct calculate, one gets that
\begin{align}
\begin{split}
q(x,t):=&(v_l-v_+)(x,t) \left[1-g_1(x+st+{\X})\right]+(v_r-v_+)(x,t)g_2 (x-st-{\X}) \\
&+V_1^{S}(x+st+{\X})+V^{S}_2(x-st-{\X} )- V_1^{S}(x+st+\beta)-V^{S}_2(x-st-\beta ) \\
\leq & C ( |v_l-  {v}_+|   +  |v_r-  {v}_+|   +|{\X} -\beta|),
\end{split}
\end{align}
and
\begin{align}
\begin{split}
\frac{\partial^{k} q}{\partial x^{k}}
\leq   C \left( \left|\frac{\partial^{k}v_{l} }{\partial x^{k}}\right|   + \left|\frac{\partial^{k}v_{r} }{\partial x^{k}} \right|   +\left|{\X} -\beta\right|\right), \quad k=1,2.
\end{split}
\end{align}
With the aid of Lemma \ref{Lem-periodic} and Lemma \ref{Lem-shift}, one gets that
\begin{equation}
\| q\|_{2}\leq C\e e^{-2\sigma_{0} t}.
\end{equation}
Similar, we obtain
\begin{equation}
\| z\|_{2}\leq C\e e^{-2\sigma_{0} t}.
\end{equation}
\end{proof}

\subsection{Low Order Estimates.}

\begin{lma}\label{f24}
Under the same assumptions of Proposition \ref{prposition3.1}, we have
\begin{align*}
\begin{split}
&\|(\phi,\Psi)\|^2(t)+\int_0^t\int_{-\infty}^{\infty} \sqrt{-\tilde{U}_x}\Psi^2 dxd t+\int_0^t\| \phi_x\|^2 dt\\
\leq& C\|(\phi_0,\Psi_0)\|^2+C{\delta} \int_0^t\|\phi_{xx}\|^{2}dt + C  e^{-c_-\beta_{1}}+C\varepsilon.
\end{split}
\end{align*}
\end{lma}
\begin{proof}
We multiply $(\ref{4.5})_1$ and $(\ref{4.5})_2$ by $\phi$ and $\frac{\Psi}{-p'(V)}$, respectively, sum them up, and  intergrading result with  respect to $t$ and $x$ over $ [0,t]\times {R}$, we have
\begin{align}
&\frac{1}{2}\int_{-\infty}^{\infty}\left(\phi^2-\frac{\Psi^2}{p'(V)}\right)dx
+\int_0^t\int_{-\infty}^{\infty}\frac{-p''(V)}{2(p'(V))^{2}} \tilde{U}_x \Psi^2
+\frac{\phi_{x}^2}{V^{\alpha+1}}dxdt\notag\\
=&\frac{1}{2}\left.\int_{-\infty}^{\infty}\left(\phi^2-\frac{\Psi^2}{p'(V)}\right)dx\right|_{t=0} \notag\\
&+\int_0^t\int_{-\infty}^{\infty}\left[G+(\alpha+1)\frac{V_x \phi_x }{V^{\alpha+2}}\right]\phi
+\frac{p(\tilde{v}|V)\Psi}{p'(V)} dxdt\label{4.9}\\
&+\int_0^t\int_{-\infty}^{\infty} F_1\phi-\frac{ \Psi}{p'(V)}\left(F_2-\frac{F_{1x}}{V^{\alpha+1}} \right) dxdt\notag\\
&+\frac{1}{2}\int_0^t\int_{-\infty}^{\infty}\frac{p''(V)}{p'(V)^{2}}(z-F_1 )_{x}\Psi^2  dxdt\notag\\
:=& \frac{1}{2}\left.\int_{-\infty}^{\infty}      \left(\phi^2-\frac{\Psi^2}{p'(V)}\right)dx\right|_{t=0}+\sum_{i=1}^3A_i.
\end{align}
By direct calculate, one gets that
\begin{equation}\label{gggf21}
\left|G+(\alpha+1)\frac{V_x \phi_x }{V^{\alpha+2}} \right|\leq C |\phi_x| (\phi_x^{2}+\phi_{xx}^{2}).
\end{equation}
Due to (\ref{4.6}), (\ref{gggf21}), we can get
\begin{align}\label{am}
\begin{split}
A_1\leq& C\int_0^t\|\phi\|_{L^\infty}\left\| \phi_x (\phi_x^{2}+\phi_{xx}^{2})\right\|_{L^1}dt+ C \int_0^t\|\Psi\|_{L^\infty}\|\phi_x^2\|_{L^1} dt\\
\leq&C \delta \int_0^t \|\phi_x\|^2 +\|\phi_{xx}\|^2 dt.
\end{split}&
\end{align}
With the aid of Lemma \ref{Lem-F}, H$\mathrm{\ddot o}$lder inequality, we have
\begin{equation}\label{344}
\begin{split}
 A_2 \leq& C \int_0^t  \|\phi,\Psi\| (\|F_{1}\|_{1} + \|F_2\|)dt\\
 \leq& C\sup_{\tau \in [0, t]}(\|\phi,\Psi\|^{2} + 1)\int_0^t \|F_{1}\|_{1} + \|F_2\|dt\\
\leq & C(\varepsilon +e^{-c_-\beta_{1}}) .
\end{split}
\end{equation}
Using H$\mathrm{\ddot o}$lder inequality, Sobolev inequality, combining  Lemma  \ref{Lem-F}, Lemma \ref{ddterm-1}, one  gets
\begin{align}\label{zhubajie5}
\begin{split}
A_3\leq &  C \int_0^t \| (z-F_1 )_{x}\|_{L^{\infty}} \|\Psi^2\|_{L^{1}}dt  \\
 \leq& C\sup_{\tau \in [0, t]}\|\Psi(\tau)\|^2 \int_0^t\| (z-F_1 )_{x}\|_{H^{1}} dt \leq    C \varepsilon.
\end{split}&
\end{align}
Inserting (\ref{am})-(\ref{zhubajie5}) into (\ref{4.9}),  using the smallness of $\delta$, we obtain the proof of Lemma \ref{f24}.
\end{proof}

\begin{lma}\label{f25}
Under the same assumptions of Proposition \ref{prposition3.1}, we have
\begin{eqnarray*}
\|(\phi,\Psi)(t)\|_{1}^2+\int_0^t\|\phi_x\|_{1}^2{d t}\leq C\|(\phi_0,\Psi_0)\|_{1}^2 + C e^{-c_-\beta_{1}}+C\varepsilon.
\end{eqnarray*}
\end{lma}
\begin{proof}
We multiply $(\ref{4.5})_1$ and $(\ref{4.5})_2 $ by $-\phi_{xx}$  and $\frac{\Psi_{xx}}{p'(V)}$, respectively and sum over the result, intergrade the result with respect to $t$ and $x$ over $[0,t]\times {R}$, we have
\begin{align}\label{zhubajie7}
& \frac{1}{2}\int_{-\infty}^{\infty}\left(\phi_x^2-\frac{\Psi_x^2}{p'(V)}\right)dx   +\int_0^t\int_{-\infty}^{\infty}\frac{-p''(V) }{2(p'(V))^{2}} \tilde{U}_x \Psi_x^2  +\frac{\phi^2_{xx}}{V^{\alpha+1}}dxdt\notag\\
=&\frac{1}{2}\int_{-\infty}^{\infty}  \left.\left(\phi_x^2-\frac{\Psi_x^2}{p'(V)}\right)dx\right|_{t=0}+\int_0^t\int_{-\infty}^{\infty}\frac{p''(V)}{p'(V)}\tilde{V}_x \Psi_{x} \phi_{x}dxdt\notag\\
&+\int_0^t\int_{-\infty}^{\infty}\frac{p(\tilde{v}|V)_{x}}{p'(V)} \Psi_{x}-G  \phi_{xx}dxdt\notag\\
 &+\int_0^t\int_{-\infty}^{\infty}\frac{p''(V)}{2p'(V)^{2}}(z-F_1 )_{x}\Psi_{x}^2+\frac{p''(V)}{p'(V)}q_x\Psi_{x}\phi_{x}dxdt\\
&-\int_0^t\int_{-\infty}^{\infty}F_1\phi_{xx}+\frac{\Psi_{x}}{p'(V)}\left(F_2-\frac{F_{1x}}{V^{\alpha+1}}\right)_{x} dxdt\notag\\
:=& \frac{1}{2}\int_{-\infty}^{\infty} \left.\left(\phi_x^2-\frac{\Psi_x^2}{p'(V)}\right) dx \right|_{t=0} +\sum_{i=1}^{4}B_i.\notag
\end{align}
With the aid of the Cauchy inequality, we have
\begin{align}\label{e3}
\begin{split}
B_1\leq&\frac{s}{4}\int_0^t \int_{-\infty}^{\infty} \frac{p''(V)}{(p'(V))^{2}} |\tilde{V}_{x}|\Psi_x^2dxdt+C \int_0^t \int_{-\infty}^{\infty}{p''(V)}|\tilde{V}_{x}|  \phi_x ^2  dxdt\\
\leq&-\frac{1}{4}\int_0^t\int_{-\infty}^{\infty} \frac{p''(V)}{(p'(V))^{2}} |\tilde{V}_{t}|{\Psi_x^2}\Psi_x^2dxdt +C\int_0^t\|\phi_x\|^2dt.
\end{split}&
\end{align}
The last inequality is based on the following inequality
\begin{align*}
\begin{split}
-\tilde{V}_t=&-(V^{S}_{1}(x+st+\beta)+V^{S}_{2}(x-st-\beta ))_{t}=-s(V^{S}_{1}(x+st+\beta)-V^{S}_{2}(x-st-\beta ))_{x} \\
>&s |V^{S}_{1x}(x+st+\beta)+V^{S}_{2x}(x-st-\beta )| = s |\tilde{V}_x|,
\end{split}&
\end{align*}
where we have used $(V^{S}_{1 })'<0$, $(V^{S}_{2 })'>0$, $s>0$.

The Cauchy inequality and the Sobolev inequality gives that
\begin{align*}
\begin{split}
B_2&\leq C\int_0^t\int_{-\infty}^{\infty}(|\phi_{xx}\phi_x|+|V_x\phi_x|)|{ \phi_{xx}}|+\left| \frac{1}{p'(V)} p(\tilde{v}|V)_{x}\Psi_{x}\right|dxdt\\
&\leq (C\delta+\eta)\int_0^t\|\phi_{xx}\|^{2}dt+ (C_{\eta}+C\delta)\int_0^t   \|\phi_{x}\|^2dt.
\end{split}&
\end{align*}
Similar like (\ref{344}) and (\ref{zhubajie5}), the error terms $B_3, B_4$ can be estimated as
\begin{equation}\label{zhuabjie6}
 B_3 + B_4 \leq C e^{-c_-\beta_{1}}+ C\varepsilon.
\end{equation}
Inserting  (\ref{e3})-(\ref{zhuabjie6}) into (\ref{zhubajie7}), we get
\begin{align*}
\begin{split}
& \frac{1}{2} \int_{-\infty}^{\infty}\left(\phi_x^2-\frac{\Psi_x^2}{p'(V)}\right)dx- \frac{1}{4}\int_0^t\int_{-\infty}^{\infty}\frac{p''(V)}{(p'(V))^{2}}|\tilde{V}_{t}| {\Psi_x^2}\Psi_x^2 dxdt+\int_0^t\int_{-\infty}^{\infty}  \frac{\phi^2_{xx}}{V^{\alpha+1}} dxdt \\
\leq& C\left(\|\phi_{0x}\|^2+\|\Psi_{0x}\|^2\right)+(C+C\delta+C_{\eta})\int_0^t\|\phi_{x}\|^2dt+(C\delta+\eta )\int_0^t\|\phi_{xx}\|^2dt\\
&+C e^{-c_-\beta_{1}}+C\varepsilon.
\end{split}&
\end{align*}
Choosing $\eta$ appropriately small and $\delta$ sufficient small, together with Lemma \ref{f24} we get  the proof of Lemma \ref{f25}.
\end{proof}

\begin{lma}\label{f27}
Under the same assumptions of Proposition \ref{prposition3.1}, we have
\begin{eqnarray*}
\int_0^t\|\Psi_{x}(t)\|^2dt\leq C\|(\phi_0,\Psi_0)\|_{1}^2+Ce^{-c_-\beta_{1}}+C\varepsilon.
\end{eqnarray*}
\end{lma}
\begin{proof}
We multiply $(\ref{4.5})_1$  by $\Psi_{x}$  and make use of $(\ref{4.5})_2$, we get
\begin{eqnarray}\label{b6}
\Psi_{x}^{2}=-\Psi_{x}G -\Psi_{x}F_1 -\frac{\Psi_{x}\phi_{xx}}{V^{\alpha+1}}+(\phi\Psi_{x})_{t}-
\phi\left[(p(V)-p(\widetilde{v})+F_{2}-\frac{F_{1x}}{V^{\alpha+1}}\right]_{x}.
\end{eqnarray}
Intergrade $ (\ref{b6})$ with  respect to $t$ and $x$ over $ [0,t]\times{R}$, we have
\begin{align*}
\begin{split}
&\int_{0}^{t}\int_{-\infty}^{\infty}\Psi_{x}^{2} dxdt\\
=&-\int_{0}^{t}\int_{-\infty}^{\infty}\Psi_{x}G dxdt
+\left.\int_{-\infty}^{\infty}\phi\Psi_{x}dx-\int_{-\infty}^{\infty} \phi\Psi_{x}dx\right|_{t=0}\\
&-\int_{0}^{t}\int_{-\infty}^{\infty}\frac{\Psi_{x}\phi_{xx}}{V^{\alpha+1}}dx  dt-\int_{0}^{t}\int_{-\infty}^{\infty}\phi_{x}\left(p(\tilde{v})-p(V)\right)dxdt\\
&+\int_{0}^{t}\int_{-\infty}^{\infty}\phi_{x}[F_{2}-\frac{F_{1x}}{V^{\alpha+1}}]-\Psi_{x}F_1    dxdt:=\sum_{i=1}^6 H_i.
\end{split}&
\end{align*}
We estimate $H_i$ term by term. By the Cauchy inequality, it follows that
\begin{align}
\begin{split}
H_1&\leq C \int_{0}^{t}\int_{-\infty}^{\infty}    \Psi_{x}(|\phi_{x}\phi_{xx}|+|V_{x}\phi_{x}|)  dxdt\\
&\leq  \eta\int_{0}^{t}  \|\Psi_{x}\|^{2} dt+ C_{\eta} \int_{0}^{t}   ( \|\phi_{xx}\|^{2} + \|\phi_{x} \|^{2})  dt.
\label{c1}
\end{split}&
\end{align}
In addition, it is straightforward to imply that
\begin{align}
\begin{split}
H_2+H_3=\int_{-\infty}^{\infty}\phi\Psi_{x}-\phi\Psi_{0x}dx
\leq   \| (\phi, \Psi_{x}) \|^{2}+ \| (\phi_{0}, \Psi_{0,x}) \|^{2},
\end{split}&
\end{align}

\begin{align}
\begin{split}
H_{4} \leq \eta\int_{0}^{t} \|\Psi_{x}\|^{2}{  dt}+C_{\eta}\int_{0}^{t} \|\phi_{xx}\|^2      dt,\quad  H_5 \leq C\int_{0}^{t}  \|\phi_{x}\|^2  dt,
\end{split}&
\end{align}

and
\begin{align}\label{c7}
\begin{split}
H_6&\leq \eta\int_{0}^{t}\|\phi_{x},\Psi_{x}\|^{2}dt+ C_{\eta}\int_{0}^{t}\| F_{2}\|^{2} +\|F_{1}\|_{1}^{2} dt\\
&\leq \eta\int_{0}^{t}\|\phi_{x},\Psi_{x}\|^{2} dt+C_{\eta}(e^{-c_-\beta_{1}}+ \varepsilon).
\end{split}&
\end{align}
Thanks to (\ref{c1})-(\ref{c7}) and Lemma \ref{f25}, taking $\eta$ sufficient small, we obtain the proof of Lemma \ref{f27}.
\end{proof}

Combining Lemma \ref{f24}-Lemma \ref{f27}, we obtain the following low-order estimate
\begin{align}\label{zhubajie77}
\|(\phi,\Psi)\|_{1}^2(t)+ \int_0^t \|\Psi_x\|^2+\|\phi_x\|_{1}^2dt\leq C\|(\phi_0,\Psi_0)\|_{1}^2 + C e^{-c_-\beta}+C\varepsilon,
\end{align}
\subsection{High Order Estimates.}
If we continue to get the   estimates of second order  derivative  $\phi_{xx}, \Psi_{xx}$, new difficulties arise.   In fact, in order to close the a priori estimate, $\|\Psi_{xx}\|_{2}$ should be sufficiently small. Unfortunately, it means that we have to add an additional condition ``$v''(0)=0$'' which can guarantee that the Dirac function will not appear. Next, we need change variables $(\phi , \Psi )$  to   $(\phi,\psi)$.
\begin{lma}
Under the same assumptions of Proposition \ref{prposition3.1}, for $0\leq t \leq T$, it holds that:
\begin{align*}
\begin{split}
\|\Psi_{0}\|_{1}^{2}\leq&\|\psi_{0}\|_{1}^{2}+C\|\phi_{0}\|_{2}^{2},
\quad\|\psi\|^{2}\leq\|\Psi\|^2+ C\|\phi\|_{1}^2,\\
\|\psi_{x}\|^{2}\leq&\|\Psi_{x}\|^2+ C\|\phi_{x}\|_{1}^2.
\end{split}
\end{align*}

\end{lma}
\begin{proof}
This lemma is similar like \cite{C2021} and the proof is omitted.
\end{proof}
 Using this lemma, low order estimate (\ref{zhubajie77}) can be rewritten as
\begin{lma}\label{lemma4.6}
Under the same assumptions of Proposition \ref{prposition3.1}, it holds that
\begin{align*}
\begin{split}
&(\|\phi\|_{1}^2 + \|\psi\|^2)(t)+\int_0^t\|\psi_x\|^2  + \| \phi_x\|_{1}^2 dt\leq C\|\phi_0\|_{2}^2+C\|\psi_0\|_{1}^2 + C e^{-c_-\beta_1}+C\varepsilon.
\end{split}
\end{align*}
\end{lma}

Next, we turn to the original equation (\ref{3.3}) to study the higher order estimates.
\begin{lma}\label{lemma4.7}
Under the same assumptions of Proposition \ref{prposition3.1}, it holds that
\begin{align}\label{4.25}
\begin{split}
&\|\psi_{x}\|^2(t)+\int_0^t\|\psi_{xx}\|^2 dt\leq C\|\phi_0\|_{2}^2+C\|\psi_0\|_{1}^2+C e^{-c_-\beta_1}+C\varepsilon.
\end{split}
\end{align}
\end{lma}
\begin{proof}  Multiplying $(\ref{3.3})_{2}$ by $-\psi_{xx}$,  integrating the result with  respect to $t$ and $x$ over $[0,t]\times{R}$ gives
\begin{align}\label{4.26}
\begin{split}
&\frac{1}{2}\|\psi_{x}\|^{2}(t)
+\int_{0}^{t}\int_{-\infty}^{\infty}\frac{{\psi_{xx}^{2}}}{V^{\alpha+1}} dxdt\\
=&\frac{1}{2}\|\psi_{0x}\|^{2}-\int_{0}^{t}\int_{-\infty}^{\infty} F_2\psi_{xx} dxdt\\
&-\int_{0}^{t}\int_{-\infty}^{\infty} f(V,U_x) \phi_{x} \psi_{xx} dxdt   -\int_{0}^{t}\int_{-\infty}^{\infty} J \psi_{xx} dxdt\\
=:&\frac{1}{2}\|\psi_{0x}\|^{2} +\sum_{i=1}^3 M_i.
\end{split}
\end{align}
Making use of Lemma \ref{Lem-F}, we have
\begin{align} \label{4.27}
\begin{split}
M_{1}&\leq \eta \int_{0}^{t}\|\psi_{xx}\|^{2}dt+ C_{\eta}\int_{0}^{t}\|F_2\|^{2}dt\\
 &\leq\eta\int_{0}^{t}\|\psi_{xx}\|^{2}dt+ C_{\eta} (e^{-c_-\beta_1} + \varepsilon).
\end{split}
\end{align}
The Cauchy inequality implies that
\begin{align}
M_{2}\leq \eta \int_{0}^{t}\|\psi_{xx}\|^{2}dt+ C_{\eta}\int_{0}^{t}  \|\phi_{x}\|^{2}dt.
\end{align}
By $(\ref{4.7})_{1}$ and the Sobolev inequality, yields
\begin{align}\label{4.29}
\begin{split}
M_{3} &\leq C \int_{0}^{t}\int_{-\infty}^{\infty}\left(\left|\phi_{x}\right|^{2}  +\left|\phi_{x}\right|\left|\psi_{xx}\right|\right)\left|\psi_{xx}\right| dxdt\\
&\leq C\int_{0}^{t}\int_{-\infty}^{\infty}\left|\phi_{x}\right|\left(\left|\phi_{x}\right|^{2}+\left|\psi_{x x}\right|^{2}\right) dxdt\\
&\leq C \delta \int_{0}^{t} \left(\left\|\phi_{x}\right\|^{2}+\left\|\psi_{x x}\right\|^{2}\right)dt.
\end{split}
\end{align}
Substituting (\ref{4.27})-(\ref{4.29}) into $(\ref{4.26})$ and using Lemma \ref{lemma4.6}, we obtain (\ref{4.25}).
\end{proof}

\begin{lma}\label{lemma4.8}
Under the same assumptions of Proposition \ref{prposition3.1}, it holds that
\begin{align}\label{4.30}
\begin{split}
&\|\phi_{xx}\|^2+\int_0^t\|\phi_{xx}\|^2 dt \leq C\|\phi_0 \|_{2}^2+C\|\psi_0\|_{1}^2+ C\delta \int_{0}^{t}\left\|\psi_{xxx}\right\|^2dt+ C e^{-c_-\beta_1}+ C\varepsilon.
\end{split}
\end{align}
\end{lma}
\begin{proof}
Differentiating $(\ref{3.3})_{1}$ with respect to $x$, using $(\ref{3.3})_{2}$, we have
\begin{align}\label{4.31}
\begin{split}
\frac{\phi_{xt}}{V^{\alpha+1}}+f(V,U_x) \phi_{x}=\psi_{t}-J+\frac{F_{1x}}{V^{\alpha+1}}-F_2.
\end{split}
\end{align}
Differentiating $(\ref{4.31})$ in respect of $x$ and multiplying the derivative by $\phi_{xx}$, integrating the result in respect of $t$ and $x$ over $[0,t]\times {R}$,   one has
\begin{align}\label{4.32}
&\frac{1}{2}\int_{-\infty}^{\infty}\frac{\phi_{xx}^{2}}{V^{\alpha+1}}dx
 +\int_{0}^{t}\int_{-\infty}^{\infty}\left(f(\tilde{V},\tilde{U}_x)+ \frac{\alpha+1}{2} \frac{\tilde{U}_x}{\tilde{V}^{\alpha+2}}\right)\phi_{xx}^{2}dxdt\notag\\
=&\frac{1}{2}\left.\int_{-\infty}^{\infty}\frac{\phi_{xx}^{2}}{V^{\alpha+1}}  dx\right|_{t=0}-\left.\int_{-\infty}^{\infty}\psi_{x} \phi_{xx}    dx\right|_{t=0}+\int_{-\infty}^{\infty}\psi_{x} \phi_{xx}dx\notag\\
&+\int_{0}^{t}\int_{-\infty}^{\infty}\left\{\frac{F_{1x}}{V^{\alpha+1}}-F_2\right\}_{x}\phi_{x x}dxdt
+\int_{0}^{t}\|\psi_{xx}\|^{2}dt-\int_{0}^{t}\int_{-\infty}^{\infty}J_{x}\phi_{xx}dxdt\\
&+(\alpha+1)\int_{0}^{t}\int_{-\infty}^{\infty}\frac{V_{x}}{V^{\alpha+2}} \psi_{xx} \phi_{xx}dxdt-\int_{0}^{t}\int_{-\infty}^{\infty}f(V,U_x)_{x}\phi_{x}\phi_{xx}dxdt\notag\\
&-\int_{0}^{t}\int_{-\infty}^{\infty}\left[f(V,U_x)-f(\tilde{V},\tilde{U}_x)+  \frac{\alpha+1}{2} \left(\frac{U_x}{V^{\alpha+2}}- \frac{\tilde{U}_x}{\tilde{V}^{\alpha+2}}\right)\right]\phi_{xx}^{2}dxdt\notag\\
=:&\frac{1}{2}\left.\int_{-\infty}^{\infty}\frac{\phi_{xx}^{2}}{V^{\alpha+1}}d x\right|_{t=0}-\left.\int_{-\infty}^{\infty}\psi_{x}\phi_{xx}    \right|_{t=0}dx+\sum_{i=1}^7 N_i.\notag
\end{align}
By $\tilde{U}_x<0$, one has
\begin{align}\label{4.33}
\begin{split}
&f(\tilde{V},\tilde{U}_x)+\frac{\alpha+1}{2}\frac{\tilde{U}_x}{V^{\alpha+2}}\\
=&-p'(\tilde{V})-\frac{\alpha+1}{2}\frac{\tilde{U}_x}{\tilde{V}^{\alpha+2}}\geq-  p'(v_-)>0.
\end{split}
\end{align}
The Cauchy inequality yields
\begin{align}
N_{1}\leq \eta \|\phi_{xx}\|^{2}+C_\eta \|\psi_{x}\|^{2}.
\end{align}
 Similar to (\ref{4.27}), we get
\begin{align}
N_{2}\leq \eta\int_{0}^{t}\|\phi_{xx}\|^{2} dt +C_\eta (e^{-c_-\beta_1} +\varepsilon).
\end{align}
$N_{3}$ can be controlled by (\ref{4.25}). Using $(\ref{4.7})_{2}$, and Cauchy inequality, we have
\begin{align*}
\begin{split}
N_{4} \leq & \eta \int_{0}^{t} \|\phi_{xx}\|^{2}dt+C_{\eta}   \int_{0}^{t}\left\|J_{x}\right\|^{2}dt\\
\leq & \eta \int_{0}^{t}\|\phi_{x x}\|^{2}dt+C_{\eta}\delta\int_{0}^{t}  \left(\left\|\phi_{x}\right\|_{1}^{2}+\left\|\psi_{x}\right\|_{2}^{2}\right)dt.
\end{split}
\end{align*}
The Cauchy inequality yields
\begin{align}
\begin{split}
N_{5}\leq C&\int_{0}^{t}\int_{-\infty}^{\infty}\left| \psi_{xx}\phi_{xx}\right| dxdt\leq \eta \int_{0}^{t} \|\phi_{xx}\|^{2}dt+C_{\eta}   \int_{0}^{t}\left\|\psi_{xx}\right\|^{2}dt.
\end{split}
\end{align}
With the help of
$$f(V,U_x)_{x}=-p^{\prime\prime}(V)V_{x}-(\alpha+1)\frac{U_{xx}}{V^{\alpha+2}}+ (\alpha+1)(\alpha+2)\frac{U_{x}}{V^{\alpha+3}}V_{x}<C,$$
one gets
\begin{align}
\begin{split}
|N_{6}|\leq &\eta \int_{0}^{t}\|\phi_{xx}\|^{2}dt+C_{\eta}\int_{0}^{t}\left\|\phi_{ x}\right\|^{2}dt.
\end{split}
\end{align}
Similar like (\ref{zhubajie5}), one gets that

\begin{align}\label{4.38}
\begin{split}
N_7\leq &  C \int_0^t \| q+z_{x}\|_{L^{\infty}} \|\phi^2\|_{L^{1}}dt  \\
 \leq& C\sup_{\tau \in [0, t]}\|\Psi(\tau)\|^2 \int_0^t\|z_{x}+q\|_{H^{1}} dt \leq    C \varepsilon.
\end{split}&
\end{align}

Choosing  $\eta$  small,  substituting  (\ref{4.33})-(\ref{4.38})  into (\ref{4.32})   and  using  Lemma \ref{lemma4.6}, Lemma \ref{lemma4.7}, we have (\ref{4.30}).
\end{proof}

On the other hand, differentiating the second equation of (\ref{3.3}) with respect to $x$, multiplying the derivative by $-\psi_{x x x}$, integrating the resulting equality over $(-\infty, \infty)  {\times}[0, t]$, using Lemma \ref{lemma4.6} - Lemma \ref{lemma4.8}, we can get the highest order estimate in the same way, which is listed as follows  and the proof is omitted.

\begin{lma}\label{lemma4.9}
Under the same assumptions of Proposition \ref{prposition3.1}, it holds that
\begin{align*}
\|\psi_{xx}(t)\|^2+\int_0^t\|\psi_{xxx}\|^2dt\leq C\|(\phi_0, \psi_0)\|_{2}^2+C e^{-c_-\beta_1}+C\varepsilon.
\end{align*}
\end{lma}
Finally, Proposition \ref{prposition3.1} is obtained by Lemma \ref{lemma4.6}-Lemma\ref{lemma4.9}.
\section{Proof of Theorem \ref{theorem}}\label{section5}
It is straightforward to imply   (\ref{2.14}) from  Lemma  \ref{lemma3.1}. It remains to show (\ref{2.15}). The following useful lemma will be used.
\begin{lma}(\cite{mn1985})\label{lemma5.1}
Assume that the function $f(t) \geq 0\in {L}^1(0, +\infty) \cap  {BV}(0, +\infty) $, then it holds that $f(t) \rightarrow0$ as $t \rightarrow \infty$.
\end{lma}
Let us turn to the system (\ref{3.3}). Differentiating  (\ref{3.3})$_1$   with respect to $x$, multiplying the
resulting equation by $\phi_{x}$ and integrating it with respect to $ {x}$ on $(-\infty,\infty)$, we have
\begin{equation*}
 \left|\frac{d}{dt}\left(\|\phi_{ x}\|^{2}\right)\right|\leq C(\|\phi_{x}\|^{2} +\|\psi_{xx}\|^{2}).
\end{equation*}
With the aid of Lemma  {\ref{lemma3.1}}, we have
\begin{equation*}
\int_{-\infty}^{\infty}\left|\frac{d}{dt}\left(\|\phi_{xx}\|^{2}\right)\right|dt \leq C,
\end{equation*}
which implies $\|\phi_{x}\|^{2}\in {L}^1(0, +\infty) \cap {BV}(0, +\infty)$. By Lemma \ref{lemma5.1}, we have
\begin{equation*}
 \|\phi_{ x}\|\rightarrow0 \quad   \text{as} \quad   t\rightarrow+\infty.
\end{equation*}
Since $\|\phi_{xx}\|$ is bounded, the Sobolev inequality implies that
\begin{eqnarray*}
\|{\tilde{v}}-V\|_{\infty}^{2}=\|\phi_{x}\|_{\infty}^{2} \leq 2\| \phi_{x}(t)\|   \|\phi_{xx}(t)\| \rightarrow  0.
\end{eqnarray*}
 Similarly, we have
 \begin{eqnarray*}
 \|{\tilde{u}}-U\|_{\infty}^{2}=\|\psi_{x}\|_{\infty}^{2} \leq 2\| \psi_{x}(t)\|   \| \psi_{xx}(t)\| \rightarrow  0.
\end{eqnarray*}
Therefore, the proof of Lemma \ref{ls} is completed.
\subsection{Proof of Theorem \ref{theorem}}

\begin{lma}\label{qusiba}
Under the assumptions   (\ref{2.16})-(\ref{2.18}), when  $\phi_{0}, \psi_{0}$ and $\beta$ satisfy
$$ \quad \beta\rightarrow  \beta_{1},\quad \|\phi_{0}\|_{H^{2}({R}_{+})}+\|\psi_{0}\|_{H^{2}({R}_{+})}\rightarrow 0, \quad\text {as} \quad \|A_0,B_0\|_{H^{2}({R}_{+})}+\beta_{1}^{-1} +\varepsilon\rightarrow 0.$$
\end{lma}
\begin{proof}

Using (\ref{2.16}) and (\ref{2.18}), we know
$\left(A_{0}, B_{0}\right) \in H^{2}(R_+)$

With the aid of
  $0<$ $-U(-s t-\beta_{1}) \leq C e^{-c-(s t+\beta_{1})}$ (see Lemma \ref{lemma2.1})  and  $\beta_{1}>0$, it follows that $\left|\int_{0}^{\infty} U(-s t-\beta_{1}) d t\right| \leq C e^{-c-\beta_{1}}$. Thus if $\beta_{1}^{-1} +\varepsilon \rightarrow0$ and $\left\|A_{0} \right\|_{H^{2}({R}_{+})} \rightarrow 0$,  using  (\ref{0507-1}), we obtain
$$
|\tilde{M}| \leq C\left(\left\|A_{0}\right\|_{H^{2}({R}_{+})}+e^{-c-\beta_{1}}+\varepsilon\right) \rightarrow 0.
$$
Similar, with the help of  (\ref{X-inf}), we have
$$|\beta- \X_{0}|\rightarrow 0.$$
Thus,  it follows that
$$|\beta-  \beta_{1}|\leq |\beta- \X_{0}|+|  \X_{0}-\beta_{1}|\leq |\beta- \X_{0}|+\frac{3}{2}|\tilde{M}|\rightarrow 0.$$

\

Set

\begin{align}
\begin{split}
(\tilde{A}_{0},\tilde{B}_0)(x)&:= -\int_{x}^{\infty} (v_{0}(y)-\zeta(y)-V_{2}^S(y-\beta), u_{0}(y)- \varphi(y)-U_{2}^S(y-\beta)  )dy,\\
\chi_{1}(x)&:=\int_{0}^{\beta_{1}-\beta}\left[v_{+}-V(x-\beta_{1}+\theta)\right] d \theta.
\end{split}
\end{align}
Make full use of (\ref{2.4}), when $|\beta_{1}-\beta|<1$, we have
$$
\left|v_{+}-V(x-\beta_{1}+\theta)\right| \leq C e^{-c+|x-\beta_{1}+\theta|} \leq C e^{-c_{+}|x-\beta_{1}|} e^{c_+|\beta_{1}-\beta|} \leq C e^{-c_+|x-\beta_{1}|}
$$
 Thus, we have
$$
\begin{aligned}
\left\|\chi_{1}\right\|^{2}_{({R}_{+})}   & \leq C \int_{0}^{\infty} (\beta_{1}-\beta)^{2} e^{-2 c_{+}|x-\beta_{1}|} d x  \leq C (\beta_{1}-\beta)^{2}
\end{aligned}
$$
where $C$ is independent of $(\beta_{1}-\beta)$ and $\beta$. Similarly, we can prove that $\left\|\chi_{1}^{\prime}\right\|_{({R}_{+})}^{2} \leq C (\beta_{1}-\beta)^{2}$ and $\left\|\chi_{1}^{\prime \prime}\right\|_{({R}_{+})}^{2} \leq C (\beta_{1}-\beta)^{2}$. Thus, we proved $\left\|\chi_{1}\right\|_{H^{2}(R_+)} \leq C|(\beta_{1}-\beta)|$.
In the same way, we have that

\begin{equation*}
\|\chi_{2} \|_{H^{2}(R_+)}:=\left\|\int_{0}^{ \beta_{1}-\beta }\left[u_{+}-U(x+\theta-\beta)\right] d \theta   \right\|_{H^{2}(R_+)} \leq C|\beta_{1}-\beta|.
\end{equation*}
Thus, we obtain
\begin{align}\label{0507-5}
\begin{split}
&\|(\tilde{A}_{0}, \tilde{B}_{0})\|_{H^{2}(R_+)} \leq\left\|\left(A_{0}, B_{0}\right)\right\|_{H^{2}(R_+)}+\left\|\left(\chi_{1}, \chi_{2}\right)\right\|_{H^{2}(R_+)}\\
 \leq & C (\left\|\left(A_{0}, B_{0}\right)\right\|_{H^{2}(R_+)}+| \beta_{1}-\beta | ).
\end{split}
\end{align}
It follows from $ |\phi_0|, |\psi_{0}|$ are all even functions that
\begin{equation}\label{gra1}
\|\phi_0\|_{H^{2}(R_+)}=\frac{1}{2}\|\phi_0\|_{H^{2}(R)}, \|\psi_{0}\|_{H^{2}(R_+)}=\frac{1}{2}\|\psi_{0}\|_{H^{2}(R)}.
\end{equation}
Using (\ref{2.5}), (\ref{4.8}) and $(\ref{{2.6}})_1$, when $x\in R_+$, one gets
\begin{align*}
&V(x,0 )-\zeta(x)-V^{S}_{2}(x  -\beta)\\
=&[V(x,0 )-\tilde{V}(x,0 )-\zeta(x)]+[\tilde{V}(x,0 )-V^{S}_{2}(x  -\beta)]\\
=&  q(x,0 )-\zeta(x)  + [V^{S}_{2}(-x  -\beta)-v_-]\\
\leq &  | \zeta(x) \left[g_2(x-\beta_{1})-g_1(x+\beta_{1})\right] | + |V^{S}_{1}(x  -\beta)-V^{S}_{1}(x  -\X)|\\
&+|V^{S}_{2}(x  +\beta)-V^{S}_{2}(x  +\X)|+ |V^{S}_{2}(-x  -\beta)-v_-|.
\end{align*}
With the aid of (\ref{0507-2}), Lemma \ref{lemma2.1} and Lemma \ref{Lem-shift}, it follows that
\begin{equation}\label{0507-3}
\|\phi_{0}-\tilde{A}_{0}\|_{H^{2}(R_+)} \leq \varepsilon+e^{-c_-\beta}.
\end{equation}
Similar, we have
\begin{equation}\label{0507-4}
\|\psi_{0}-\tilde{B}_{0}\|_{H^{2}(R_+)} \leq \varepsilon+e^{-c_-\beta}.
\end{equation}
Combining (\ref{0507-5})-(\ref{0507-4}), one gets that  $\|\phi_{0}\|_{H^{2}({R}_{+})}+\|\psi_{0}\|_{H^{2}({R}_{+})}\rightarrow 0$.
\end{proof}
Once this lemma is proved , we begin the proof of our main result. Using (\ref{2.5}), (\ref{4.8}) and $(\ref{{2.6}})_1$, when $x\in R_+$, one gets
\begin{align*}
&v(x,t)-V^{S}_{2}(x-st -\beta)\\
=&[v (x,t) -V(x,t )]+[V(x,t )-\tilde{V}(x,t )]+[\tilde{V}(x,t )-V^{S}_{2}(x-st -\beta)]\\
=&[{\widetilde{v}(x,t)}-V(x,t )]+ q(x,t )  + [V^{S}_{2}(-x-st -\beta)-v_-]\\
\leq &|{\widetilde{v}(x,t)}-V(x,t )|+ |q(x,t ) | + |V^{S}_{2}(-x-st -\beta)-v_-|.
\end{align*}
We obtain that
\begin{align*}
\begin{split}
&\|v(x,t)-V^{S}_{2}(x-st -\beta)\|_{L^{\infty}}\\
\leq &\|{\widetilde{v}(x,t)}-V(x,t )\|_{L^{\infty}}+ \|q(x,t )\|_{L^{\infty}}  +  \|V^{S}_{2}(-x-st -\beta)-v_-\|_{L^{\infty}}.
\end{split}
\end{align*}
Together with  (\ref{gra1}), Lemma \ref{lemma2.1}-Lemma \ref{ls} and Lemma \ref{qusiba} we obtain the proof of  Theorem \ref{theorem}.

\section{Proof of Lemma \ref{Lem-shift} and Lemma \ref{Lem-F} }\label{Sec-shift-F}
\subsection{Proof of Lemma \ref{Lem-shift}}
\begin{proof}\color{black}

By Lemma \ref{Lem-periodic}, we have $|{\X'(t)}|, |{\Y'(t)}| \leq C\e e^{-2\sigma_{0}t}$ for all $t>0$. Thus $\lim\limits_{t\rightarrow +\infty} \X(t)$ and $\lim\limits_{t\rightarrow +\infty}\Y(t)$ are all exist. In the following part of this subsection, we compute the two limits. Motivated by \cite{HY1shock}, we define the domain
\begin{equation}
\begin{cases}
\Omega_{y}^N(t):= \left\{ (x,\tau): 0<\tau<t,\quad \Gamma_l^N(\tau) < x < \Gamma_r^N(t) \right\},\\
\Gamma_l^N(\tau):= -s\tau-\X(\tau)+(-N+y) \pi,\\
\Gamma_r^N(\tau):= s\tau+\X(\tau) +(N+y)\pi,
\end{cases}
\end{equation}
where $y\in [0,1], N \in N^*$.
Using $(\ref{3.1})_1$, we have
$$\lim_{N\rightarrow +\infty}\int_{0}^{1}\iint_{\Omega_{y}^N(t)} (V_t-U_x) dxd\tau dy=0.$$
With the aid of Green   formula, one gets
\begin{align}\label{integral}
\begin{split}
\lim_{N\rightarrow +\infty} \int_{0}^{1}   \mathfrak{f}(N,y) dy=0.
\end{split}&
\end{align}
where
\begin{align*}
\begin{split}
\mathfrak{f}(N,y)=&\int_{\Gamma_l^N(0)}^{\Gamma_r^N(0)} V(x,0)dx+\int_0^t[(s +\X')V+U ] (\Gamma_r^N (\tau),\tau) d\tau\\
&-\int_{\Gamma_l^N(t)}^{\Gamma_r^N(t)} V(x,t)dx
-\int_0^t[(-s -\X')V+U] (\Gamma_l^N(\tau),\tau) d\tau:=\sum_{i=1}^{4}S_i(N,y).
\end{split}&
\end{align*}
We rewrite $S_1+S_3$ as:
\begin{align*}
&S_1+S_3=\sum_{i=1}^4 I_i,
\end{align*}
where
\begin{align*}
I_1&=\int_{\Gamma_l^N(0)}^{\Gamma_r^N(0)}\zeta(-x)(1-g_1(x+\X_0))+\zeta(x)  g_2(x-\X_0)dx,\\
I_2&=\int_{\Gamma_l^N(0)}^{\Gamma_r^N(0)}V^{S}_{1}(x+{\X_0})-{v}_-+ V^{S}_{2}(x-{\X_0})dx,\\
I_3&=-\int_{\Gamma_l^N(t)}^{\Gamma_r^N(t)}\zeta_{l}(x,t)(1-g_1(x+st+{\X}))+ \zeta_{r}(x,t)g_2  (x-st-{\X})dx,\\
I_4&=-\int_{\Gamma_l^N(t)}^{\Gamma_r^N(t)}V^{S}_{1}(x+st+{{\X}})-{v}_-+V^{S}_{2}(x-st-{{\X}})dx.
\end{align*}
Here
\begin{align*}
\zeta_{i}=v_{i}-\bar{v}_{i}; \quad i=l,r.
\end{align*}
Moreover, $I_1$ can be rewrite as
\begin{align*}
I_1=&\int_{0 }^{\Gamma_r^N(0)}[\zeta(-x) (1-g_1(x+\X_0))-\zeta(x) (1- g_2(x-\X_0)) ] dx+\int_0^{\Gamma_r^N(0)}\zeta(x)dx,\\
&+ \int_{\Gamma_l^N(0) }^{0}[-\zeta(-x)g_1(x+\X_0)+\zeta(x)g_2(x-\X_0) ]dx+ \int_{\Gamma_l^N(0)}^0 \zeta(-x)dx.
\end{align*}
Since $\int_{0}^{\pi}\zeta(x) dx=0  $, then
\begin{align*}
\int_{0}^{1}\int_{0}^{\X_0 + y \pi} \zeta (x) dxdy &=\frac{1}{\pi}\int_{0}^{\pi} \int_{0}^{\X_0 +z}\zeta (x)dxdz=\frac{1}{\pi}\int_{0}^{\pi} \int_{0}^{y}\zeta (x)dxdy,\\
\int_{0}^{1}\int_{\X_0+y\pi}^0 \zeta (-x) dxdy&=-\frac{1}{\pi}\int_{0}^{\pi} \int_{0}^{\X_0 +z}\zeta (-x)dxdz=-\frac{1}{\pi}\int_{0}^{\pi} \int_{0}^{y}\zeta (-x)dxdy.
\end{align*}
So we obtain
\begin{align}\label{I-1}
\lim_{N\rightarrow +\infty} \int_0^1I_1 dy =&\int_{0 }^{\infty}\zeta(-x)(1- g_1(x+\X_0))-\zeta(x)(1-g_2(x-\X_0))dx,\notag\\
&+\int_{-\infty}^{0}-\zeta(-x)g_1(x+\X_0)+\zeta(x)g_2(x-\X_0)dx,\notag\\
&+\frac{1}{\pi}\int_{0}^{\pi} \int_{0}^{y}\zeta (x)-\zeta (-x)dxdy,\notag\\
=&2\int_{0 }^{\infty}\zeta(-x)(1- g_1(x+\X_0))-\zeta(x)(1-g_2(x-\X_0))dx,\notag\\
&+\frac{1}{\pi}\int_{0}^{\pi} \int_{0}^{y}\zeta (x)-\zeta (-x)dxdy,
\end{align}
where we have used (\ref{2.5}),(\ref{2.7})in the last equality. With the aid of Lemma \ref{Lem-periodic}, one gets that
\begin{align}\label{I-3}
\left|\lim_{N\rightarrow +\infty}\int_0^1I_3dy\right|\leq C e^{-2\sigma_{0} t}.
\end{align}
By directly calculate, we have
\begin{align*}
I_2+I_4=&-v_-[\Gamma_r^N(0)-\Gamma_l^N(0)]+v_-[\Gamma_r^N(t)-\Gamma_l^N(t)]\\
-&\int_{(N+y)\pi+2\X_{0}}^{2st +2{\X}+(N+y)\pi}V^S_1 (x)dx
+\int_{(-N+y)\pi-2\X_{0}}^{-2st- 2{\X}+(-N+y)\pi} V^S_2(x)dx.
\end{align*}
Using  (\ref{2.3}) (\ref{2.5}), one gets that
\begin{align}\label{I-2-4}
\lim_{N\rightarrow +\infty} \int_0^1(I_2+I_4)dy=-2v_-(s t+{\X}-\X_{0}).
\end{align}	
The integral on $\Gamma_r^N$ in (\ref{integral}) satisfies that
\begin{align}\label{int-Ga-r}
&\lim_{N\rightarrow +\infty} \int_{0}^{1}  S_2 (N,y)dy \notag \\
=&\lim_{N\rightarrow +\infty} \int_{0}^{t} \int_{0}^{1}  [(s +\X') v_r + u_r ] ( \Gamma_r^N(\tau), \tau ) dy d\tau \notag \\
=&(s t +{\X}-\X_0)v_++u_+ t.
\end{align}
Here we have used Since $(V,U)\rightarrow(v_r,u_r)$ as $x\rightarrow +\infty$.
By same method, we obtain
\begin{equation}\label{int-Ga-l}
\lim_{N\rightarrow +\infty} \int_{0}^{1}  S_4 (N,y) dy =- [(-s  t -{\X}+\X_0)v_+  -u_+ t].
\end{equation}
Collecting (\ref{integral})-(\ref{int-Ga-l}), it follows   that
\begin{align*}
& 2\int_{0}^{+\infty} \left[ \zeta(-x) \left(1-g_1(x+\X_0)\right)-\zeta(x) \left(1-g_2(x-\X_0)\right) \right]dx\\
\quad+&\frac{1}{\pi}\int_{0}^{\pi} \int_{0}^{y}\zeta (x)-\zeta (-x)dxdy\\
\quad  +& 2 (v_+-v_-)(st+\X-\X_0)+2u_+t=O( e^{-2\sigma_{0}t}),
\end{align*}
Thus we obtain (\ref{X-inf}) where we have used  R-H conditions (\ref{rh})$ _1$. We omit the proof of (\ref{Y-inf}),  since    it is similar with (\ref{X-inf}).
\end{proof}

\subsection{Proof of Lemma \ref{Lem-F}}
We only give  the proof of $F_1$, due to the fact that  the proof of $F_2$ is similar.

Case 1. For $ x< s t, $ we rewrite $ F_1(x,t) $  as follows.
\begin{equation*}
F_1(x,t)  := D_{1,1}^-(x,t) + D_{1,2}^-(x,t),
\end{equation*}
 where\begin{align*}
D_{1,1}^-(x,t) := &\theta \left\{ (-\X') g_{1}(x+st+{\X}) + (-s) \left(g_{1}(x+st+{\X})-g_{1}(x+st+{\Y})\right) \right\} \\
 &-\theta \left\{ (\X') g_{2}(x-st-\X) + s \left(g_{2}(x-st-\X )-g_{2}(x-st-{\Y})\right) \right\}, \\
D_{1,2}^-(x,t) := &  \zeta_l (x ,t) \left[ g_{1}(x+st+\Y)-g_{1}(x+st+\X)\right]+ \int_{-\infty}^x \varphi_l (y,t) g'_{1}(y+st+\X) dy \\
& +(s+\X') \int_{-\infty}^x  \zeta_l  (y,t) g'_{1}(y+st+\X) dy\\
  & - \zeta_r  (x ,t) \left[ g_{2}(x-st-\Y)-g_{2}(x-st-\X)\right]- \int_{-\infty}^x \varphi_r (y,t) g'_{2}(y- st-\X) dy \\
& -(s+\X') \int_{-\infty}^x  \zeta_r  (y,t) g'_{2}(y- st-\X) dy.
\end{align*}
It follows that
 \begin{align}\label{zhubajie1999}
 \begin{split}
&\sum_{k=0}^2 \int_{-\infty}^{s  t}  | \p_x^k F_1(x,t) |^2 dx \\
\leq &C \e^2 e^{-4 \sigma_{0} t} \sum_{k=0}^2 \int_{-\infty}^{s  t}   \left|  g_1^{(k)}  (x+st+{\X}) \right|^2+   \left| { g_2^{(k)}} (x-st-{\X}) \right|^2  dx \\
\leq &C \e^2 e^{-4\sigma_{0} t} \int_{-\infty}^{-s t} \theta( e^{-\theta |{x+st}| } +  e^{-\theta |{x-st}| } ) dx+  \int_{-st}^{st} (1+\theta e^{-\theta |{x-st}| } ) dx   \\
\leq & C\e^2 e^{-4\sigma_{0} t} \big[ C+2st \big]\leq  C\e^2 e^{-2\sigma_{0} t}.
\end{split}&
\end{align}
Here we have used Lemma \ref{Lem-periodic}, (\ref{2.4}), (\ref{2.5}), (\ref{2.7}). Moreover, $F_{2}$  can be  rewritten as as follows.
\begin{align}\label{zhubajie8}
\begin{split}
F_{2}:=W+ D_{2,1}^-(x,t) + D_{2,2}^-(x,t),,
\end{split}
\end{align}
where
\begin{align}\label{zhubajie9}
\begin{split}
W :=&p(\tilde{V})+p(v_-)-p(V^{S}_{1}(x+st+{\beta}) )-p(V^{S}_2(x-st-{\beta} ))\\
  &+\frac{U^{S}_{2x}(x-st-{\beta} )}{V^{S}_2(x-st-\beta )^{\alpha+1}}  +\frac{U^{S}_{1x}(x+st+{\beta} )}{V^{S}_{1}(x+st+\beta )^{\alpha+1}}- \frac{\tilde{U}_{x}}{\tilde{V}^{\alpha+1}},
\end{split}
\end{align}

\begin{align}
 D_{2,1}^-(x,t) : =
  &-\left[p(V^{S}_{1}(x+st+{\Y}))-p(V^{S}_{1}(x+st+{\beta} ))\right]\notag\\
  &-\left[p(V^{S}_2(x-st-{\Y})) -p(V^{S}_2(x-st-{\beta} ))\right]\notag\\
  &+\left[\frac{U^{S}_{2x}(x-st-{\Y})}{V^{S}_2(x-st-{\Y})^{\alpha+1}}-\frac{U^{S}_{2x}(x-st-{\beta} )}{V^{S}_2(x-st-\beta )^{\alpha+1}}\right]\notag\\
    &+\left[\frac{U^{S}_{1x}(x+st+{\Y})}{V^{S}_{1}(x+st+{\Y})^{\alpha+1}}  -\frac{U^{S}_{1x}(x+st+{\beta} )}{V^{S}_{1}(x+st+\beta )^{\alpha+1}}\right]\notag\\
& -\left[\frac{U_{x}}{V^{\alpha+1}}
-\frac{\tilde{U}_{x}}{\tilde{V}^{\alpha+1}}\right]+p(V)-p(\tilde{V}),\notag
\end{align}
and
\begin{align}
 D_{2,2}^-(x,t) : =&\int_{-\infty} ^{x}\big[ -s (u_l+ u_+ )-\Y' u_l +p(v_+) - p(v_l ) +   \frac{ u_{lx}}{v_l^{\alpha+1}}  \big] g'_1(x+st+{\Y})  ) dx \notag\\
&-\int_{-\infty} ^{x} \big[ s (u_r -  u_+) +\Y' u_r - p(v_r )+p(v_+) +   \frac{ u_{rx}}{v_r^{\alpha+1}} \big]  g_2'(x-st-{\Y}) dx.\notag
\end{align}

\begin{align}
\begin{split}
|W|\leq& \left| \left(\frac{1}{(V^{S}_{1}(x+st+\beta ))^{\alpha+1}}-\frac{1}{\tilde{V}^{\alpha+1}}\right) U^{S}_{1x}(x+st+\beta ) \right|\\
&+\left|\left( \frac{1}{(V^{S}_{2}(x-st-\beta ))^{\alpha+1}}-\frac{1}{\tilde{V}^{\alpha+1}}\right) U^{S}_{2x}(x-st-\beta )\right|\\
&+\left| p(V^{S}_{1}(x+st+\beta )+V^{S}_{2}(x-st-\beta )-v_{-})-p(V^{S}_{1}(x+st+\beta ))\right|\\
&+\left|p(v_-)-p(V^{S}_{2}(x-st-\beta )) \right|\\
\leq&C\{ |(V^{S}_{2}(x-st-\beta )-v_-)|+|U^{S}_{2x}(x-st-\beta )|\}.
\end{split}&
\end{align}
By $(\ref{2.1}) $,  we get
\begin{align}\label{6.16}
\begin{split}
 &\left|\frac{\partial^{j}U^{S}_{2}(x-st-\beta )}{\partial x^{j}}\right| ,\left|\frac{\partial^{j}(V^{S}_{2}(x-st-\beta )-v_-)}{\partial x^{j}}\right| \\
    \leq C  &|V^{S}_{2}(x-st-\beta )-v_-|, \forall j\in \mathbb{N}.
\end{split}&
\end{align}
On the other hand,  in the same way,   it is still true to replace ($V^{S}_{2}(x-st-\beta ),U^{S}_{2}(x-st-\beta )$) with ($V^{S}_{1}(x+st+\beta ),U^{S}_{1}(x+st+\beta )$) in (\ref{6.16}).  We get $\left|\frac{\partial^{n}W}{\partial x^{n}}\right| \leq C  |V_i^{S}(x+(-st-\beta)^{i+1} )-v_-|, i=1,2;\forall n\in \mathbb{N}.$
  If we choose $\beta> 0$ sufficiently large, for $n=0,1 $, it follows that:
\begin{align*}
\begin{split}
\int_{-\infty}^{\infty} \left|\frac{\partial^{n} W}{\partial{x}^{n}}\right|^{2}  dx=& \int_{-\infty}^{ 0}\left|\frac{\partial^{n} W}{\partial{x}^{n}}\right|^{2}  \operatorname{d} x+ \int_{0}^ { \infty}\left|\frac{\partial^{n} W}{\partial{x}^{n}}\right|^{2} dx\\
 \leq &C \int_{-\infty }^{0}   |V_{2}( x-st-\beta)-v_{-}|^{2} dx+C \int_{ 0}^ { \infty }  |V_{1}( x+st+\beta )-v_{-}|^{2} dx\\
 \leq &C \theta^{2} \int_{-\infty }^{ 0}    \exp[ 2c_{-}( x-s t -\beta)]   dx+C \theta^{2} \int_{ 0}^ { \infty }   \exp[ -2c_{-}  ( x+s t+\beta )]   dx\\
  \leq &  C  e^{  -2c_{-} s t   }   e^{- 2c_{-} \beta }= C  e^{  -2c_{-} s t   }   e^{- 2c_{-} \beta_{1} }  e^{- 2c_{-} (\beta-\beta_{1}) }\leq C  e^{  -2c_{-} s t   }   e^{- 2c_{-} \beta_{1} }  ,
  \end{split}&
\end{align*}
where we have used Lemma \ref{lemma2.1} in the second inequality and Lemma \ref{qusiba} in the last inequality. Thus, we obtain that
\begin{flalign*}
 \|W\|_{{2}}\leq C  e^{-c_-\beta_1}e^{- s c_-t}.
\end{flalign*}
Similar like (\ref{zhubajie1999}), one can get that
\begin{equation}
\sum_{k=0}^2 \sum_{i=1}^2 \int_{-\infty}^{s  t}  | \p_x^k  D_{2,i}^-(x,t) |^2 dx \leq  C\e^2 e^{-2\sigma_{0} t}.
\end{equation}

Case 2. If $ x > s t$, using (\ref{ode-shift}), one can decompose $ F_1, F_2$ as
 \begin{equation}
F_1(x,t) = -F_{1,1}(x,t) + \int_{x}^{+\infty} f_{1,2}(y,t) dy +\X' \int_{x}^{+\infty} F_{1,3}(y ,t)dy,
\end{equation}
 \begin{equation}
F_2(x,t) = -F_{2,1}(x,t) + \int_{x}^{+\infty} f_{2,2}(y,t) dy +\Y' \int_{x}^{+\infty} F_{2,3}(y ,t)dy,
\end{equation}
 and using similar arguments as in the case 1 to obtain that
\begin{equation}
\sum_{k=0}^2\int_{st}^{+\infty}|\p_x^k F_i(x,t)|^2 dx\leq C\e^2 e^{-2\sigma_{0} t}, \quad i=1,2.
\end{equation}

\begin{Remark}
If $\gamma=1$ (isothermal  gas) in our equations,  we can get the same result by   the same method.
\end{Remark}
\begin{Remark}
In our proof, we make the position of the shock is far away from the wall, is this necessary?
\end{Remark}

\end{document}